\documentclass[a4paper,12pt]{amsart}

\usepackage{amsfonts}
\usepackage{amsmath}
\usepackage{amssymb}

\usepackage{mathrsfs}

\usepackage[colorlinks]{hyperref}
\usepackage{hyperref}
\renewcommand\eqref[1]{(\ref{#1})} 

\numberwithin{equation}{section}
\theoremstyle{plain}
\newtheorem{thm}{Theorem}[section]

\newtheorem{cor}[thm]{Corollary}
\newtheorem{lemma}[thm]{Lemma}

\theoremstyle{definition}
\newtheorem{defi}[thm]{Definition}

\def\p#1{{\left({#1}\right)}}

\def\supp{\textrm{supp}}

\def\C{{\mathbb C}}
\def\A{{\mathbb A}}
\def\Rr{{\mathbb R}}

\def\Sp{{\mathcal H}}
\def\ind{\mathcal I}
\def\Dcal{{\mathcal D}}
\def\H{{\mathcal L}}

\newcommand{\eps}{\varepsilon}
\newcommand{\esp}{\mathrm{e}}

\def\Dis{{\mathrm H}^{-\infty}}
\def\Cs{{\mathrm H}}

\newcommand{\beq}{\begin{equation}}
\newcommand{\eeq}{\end{equation}}

\begin{document}

\title[Wave equation for operators with discrete spectrum]
{Wave equation for operators with discrete spectrum and irregular propagation speed}

\author[Michael Ruzhansky]{Michael Ruzhansky}
\address{
  Michael Ruzhansky:
  \endgraf
  Department of Mathematics
  \endgraf
  Imperial College London
  \endgraf
  180 Queen's Gate, London, SW7 2AZ
  \endgraf
  United Kingdom
  \endgraf
  {\it E-mail address} {\rm m.ruzhansky@imperial.ac.uk}
  }
\author[Niyaz Tokmagambetov]{Niyaz Tokmagambetov}
\address{
  Niyaz Tokmagambetov:
  \endgraf
    al--Farabi Kazakh National University
  \endgraf
  71 al--Farabi ave., Almaty, 050040
  \endgraf
  Kazakhstan,
  \endgraf
   and
  \endgraf
    Department of Mathematics
  \endgraf
  Imperial College London
  \endgraf
  180 Queen's Gate, London, SW7 2AZ
  \endgraf
  United Kingdom
  \endgraf
  {\it E-mail address} {\rm n.tokmagambetov@imperial.ac.uk}
 }

\thanks{The first author was supported in parts by the EPSRC grant EP/K039407/1 and by the Leverhulme Grant RPG-2014-02. No new data was collected or generated during the course of research.}

\date{\today}


\maketitle

\begin{abstract}
Given a Hilbert space $\Sp$, we investigate the well-posedness of the Cauchy problem for the wave equation for operators with discrete non-negative spectrum acting on $\Sp$.
We consider the cases when the time-dependent propagation speed is regular, H\"older, and distributional. We also consider cases when it it is strictly positive (strictly hyperbolic case) and when it is non-negative (weakly hyperbolic case). When the propagation speed is a distribution, we introduce the notion of ``very weak solutions" to the Cauchy problem. We  show that the Cauchy problem for the wave equation with the distributional coefficient has a unique ``very weak solution" in appropriate sense, which coincides with classical or distributional solutions when the latter exist. Examples include the harmonic oscillator and the Landau Hamiltonian on $\mathbb R^n$, uniformly elliptic operators of different orders on domains, H\"ormander's sums of squares on compact Lie groups and compact manifolds, operators on manifolds with boundary, and many others.
\end{abstract}

\section{Introduction}

Let $\H$ be a densely defined linear operator with the discrete spectrum
$\{\lambda_{\xi}\geq 0: \, \xi\in\ind\}$ on the Hilbert space $\Sp$. The main (and only) assumption in this paper will be that the system of
eigenfunctions
$\{e_{\xi}: \; \xi\in\ind\}$ is a basis in $\Sp$,
where $\ind$ is a countable set.

Note that we do not assume a-priori that the operator $\H$ is self-adjoint in any sense so that the basis
$\{e_{\xi}: \; \xi\in\ind\}$ does not have to be orthogonal.

In this paper, for a non-negative function $a=a(t)\geq 0$ and for the source term $f=f(t)\in\Sp$, we are interested in the well-posedness of the Cauchy problem for the operator $\H$ with the propagation speed given by $a$:
\begin{equation}\label{CPa}
\left\{ \begin{split}
\partial_{t}^{2}u(t)+a(t)\H u(t)&=f(t), \; t\in [0,T],\\
u(0)&=u_{0}\in\Sp, \\
\partial_{t}u(0)&=u_{1}\in\Sp.
\end{split}
\right.
\end{equation}
In order to provide a comprehensive analysis of such problems, we will treat several cases depending on the properties of $a$, and to a lesser extent of $f$. The reason behind this is that in each case the optimal results that we can get are different and depend on the properties of $a$. More specifically, we consider the following cases:
\begin{itemize}
\item[(I.1)] The coefficient $a$ and the source term $f$  are regular enough: $a\in C^{1}, f\in C$, and $a\geq a_{0}>0$. This is the classical case where we show the (natural) well-posedness of \eqref{CPa} in Sobolev spaces associated to $\H$.
\item[(I.2)] We consider the case when $a\in C^{\alpha}$, $0<\alpha<1$, $a\geq a_{0}>0$, is strictly positive and H\"older of order $\alpha$. In this case it is well-known already for $\H u=-u''$ on $\mathbb R$ that the Cauchy problem may be not well-posed in $C^{\infty}$ or in $\mathcal D'$
(see e.g. Colombini, de Giorgi and Spagnolo \cite{Colombini-deGiordi-Spagnolo-Pisa-1979} or \cite{Colombini-Jannelli-Spagnolo:Annals-low-reg}) and the Gevrey spaces appear naturally. Here we prove the well-posedness of \eqref{CPa} in the scale of $\H$-Gevrey spaces and $\H$-ultradistributions that we introduce for this purpose.
\item[(I.3)]  We consider the case when $a\in C^{\ell}$, $\ell\geq 2$, $a\geq 0$, is regular but may be equal to zero (the weakly hyperbolic case). In this case there may be also no well-posedness in  $C^{\infty}$ or in $\mathcal D'$ already for $\H u=-u''$ on $\mathbb R$. Here we also prove
the well-posedness of \eqref{CPa} in the scale of $\H$-Gevrey spaces and $\H$-ultradistributions.
\item[(I.4)] The last `regular' case is the weakly hyperbolic case with H\"older propagation speed: when $a\in C^{\alpha}$, $0<\alpha<2$, $a\geq 0$.
Here we also prove
the well-posedness of \eqref{CPa} in the scale of $\H$-Gevrey spaces and $\H$-ultradistributions depending on $\alpha$.
\end{itemize}
Consequently, we also consider the cases when $a$ is less regular than H\"older, allowing it to be a (positive) distribution, for example allowing the case
$$ a=1+\delta,$$
involving the $\delta$-distribution. Such type of setting appears in applications, for example when one is looking at the behaviour of a particle in irregular electromagnetic fields: in this case $\H$ is the Landau Hamiltonian on $\mathbb R^{n}$, and the corresponding wave equation was analysed by the authors in \cite{RT16a}. While from the physical point of view (of irregular electromagnetic fields) such situation is natural and one expects the well-posedness, mathematically equation \eqref{CPa} is difficult to handle because of the general impossibility to multiply distributions (recall the famous Schwartz impossibility result from \cite{Schwartz:impossibility-1954}).

In the setting of $\H$ being a second order invariant partial differential operator in $\mathbb R^{n}$, in \cite{GR15b}, Claudia Garetto and the first-named author introduced the notion of ``very weak solutions'', proving their existence, uniqueness, and consistency with classical or distributional solutions should the latter exist, for wave-type equations in $\mathbb R^{n}$. The setting of the present paper is different (since we assume that $\H$ has a discrete spectrum), and in \cite{RT16a} the authors proved the existence, uniqueness, and consistency for the case when $\H$ is the Landau Hamiltonian on $\mathbb R^{n}$ (see the example in Subsection \ref{SEC:LH2d}). Thus, the second aim of this paper is to develop the general notion of very weak solutions for the abstract problem \eqref{CPa}. In particular case of the Landau Hamiltonian, the results of this paper also extend those in \cite{RT16a} by allowing a wider class of admissible Cauchy data $u_{0},u_{1}$.
The analysis of very weak solutions is based on the results and techniques of cases (I.1)-(I.4).
Thus, in this paper we also consider the following situations:
\begin{itemize}
\item[(II.1)] The coefficient $a\geq a_{0}>0$ is a strictly positive distribution and the Cauchy data $u_{0},u_{1}$ and the source term $f(t)$ belong to the $\H$-Sobolev spaces $H_{\H}^{s}$ for some $s\in\mathbb R$. In this case we prove the existence and uniqueness of Sobolev-type very weak solutions, and their consistency with cases (I.1)-(I.4) when we know that stronger solutions exist.
\item[(II.2)] The coefficient $a\geq 0$ is a non-negative distribution and the Cauchy data $u_{0},u_{1}$ and $f(t)$ are $\H$-distributions or $\H$-ultradistributions. In this case we prove the existence and uniqueness of ultradistributional-type very weak solutions, and their consistency with cases (I.2)-(I.4) when we know that ultradistributional solutions exist.
\end{itemize}

We divide the presentation of our results in two parts for the cases (I.1)-(I.5) and (II.1)-(II.2), respectively.

\smallskip
We note that we can partially remove the condition that the spectrum $\lambda_{\xi}\geq 0$ is non-negative. Indeed, let $\H_{0}$ be a densely defined linear operator with the discrete spectrum
$\{\lambda_{\xi}\in\mathbb C: \, \xi\in\ind\}$ on the Hilbert space $\Sp$, and assume that the system of corresponding eigenfunctions
$\{e_{\xi}: \; \xi\in\ind\}$ is a basis in $\Sp$,
where $\ind$ is an ordered countable set. We denote by $\H:=|\H_{0}|$ the operator defined by
assigning the eigenvalue $|\lambda_{\xi}|$ for each eigenfunction $e_{\xi}$. Moreover, if $\lambda_{\xi}=0$ for some $\xi$, for example to define negative powers of an operator, we can put $\H:=|\H_{1}|$ to be the operator defined by the eigenvalue $(|\lambda_{\xi}|+c)$ to each eigenfunction $e_{\xi}$, with some positive $c>0$.  We note that $\H$ is not the absolute value of $\H_{0}$ in the operator sense since $\H_{0}$ and its adjoint $\H_{0}^{*}$ may have different domains and are, in general, not composable. However, this is well-defined by the symbolic calculus developed in \cite{RT16} (and extended in \cite{RT16b} to the full pseudo-differential calculus without the condition that eigenfunctions do not have zeros).
Therefore, all the results of the paper extend to the Cauchy problem
\begin{equation}\label{CPa0}
\left\{ \begin{split}
\partial_{t}^{2}u(t)+a(t)|\H_{0}| u(t)&=f(t), \; t\in [0,T],\\
u(0)&=u_{0}\in\Sp, \\
\partial_{t}u(0)&=u_{1}\in\Sp,
\end{split}
\right.
\end{equation}
if we apply the results for \eqref{CPa} taking $\H=|\H_{0}|$ in the above sense.

The organisation of the paper is as follows. In Section \ref{SEC:results} we formulate the results for the cases (I.1)-(I.5). In Section \ref{SEC:examples} we give examples of different settings with different operators $\H$ satisfying our assumptions (discrete spectrum and a basis of eigenfunctions).
In Section \ref{Sec:part2} we formulate the results for cases (II.1)-(II.2) corresponding to propagation speeds of low regularity.
In Section \ref{SEC:Prelim} we review elements of the (nonharmonic) Fourier analysis associated to $\H$.
In Section \ref{SEC:reduction} we prove results of Part I from Section \ref{SEC:results} and in Section \ref{SEC:proofs2} we prove results of Part II from Section \ref{Sec:part2}.

\section{Main results, Part I}
\label{SEC:results}


\vspace{1mm}

In our results below, concerning the Cauchy problem \eqref{CPa}, we first
carry out analysis in the strictly hyperbolic case $a(t)\ge a_0>0$, $a\in C^1([0,T])$.
This is the regular strictly hyperbolic type case when we obtain the well-posedness
in Sobolev spaces ${H}^{s}_\H$ associated to the operator $\H$:
for any $s\in\Rr$, we set
\begin{equation}\label{EQ:Sob}
H^s_\H:=\left\{ f\in\Dis_{\H}: \H^{s/2}f\in
\Sp\right\},
\end{equation}
with the norm $\|f\|_{H^s_\H}:=\|\H^{s/2}f\|_{\Sp}.$
The global space of $\H$-distributions $\Dis_{\H}$ is defined in
Section \ref{SEC:Prelim}. It is notationally more convenient to use the operator $\H^{1/2}$ in \eqref{EQ:Sob} because the operator in \eqref{CPa} is second order with respect to $t$: $\H$ is positive so $\H^{1/2}$ is well defined by its spectral decomposition, but in Section \ref{SEC:Prelim} we will also make a symbol definition of $\H^{1/2}$. Namely, $\sigma_{\H^{1/2}}(\xi)=\lambda_{\xi}^{1/2}$.
Anticipating the material of the next sections, using Lemma \ref{LEM: FTl2} and
Plancherel's identity \eqref{EQ:Plancherel}, in our case we can express the Sobolev norm as
\begin{equation}\label{EQ:Sob1}
\|f\|_{H^s_\H}= \p{ \sum_{\xi\in\ind} |\lambda_{\xi}|^{s}
\left| (f, e_{\xi})\right|^2}^{1/2},
\end{equation}
for any $s\in\mathbb R$,
where $(\cdot, \cdot)$ is the inner product of $\Sp$.

Moreover, if $\lambda_{\xi}=0$ for some $\xi$, in order to define negative powers of an operator, for eigenvectors corresponding to the zero eigenvalue, we can, without loss of generality, for some positive $c>0$, redefine $\H$ to be the operator assigning the eigenvalue $(|\lambda_{\xi}|+c)$ to each eigenfunction $e_{\xi}$.

\begin{thm}[Case I.1]
\label{theo_case_1}
Assume that $a\in C^1([0,T])$ and that $a(t)\ge a_0>0$.
For any $s\in\Rr$, if $f\in C([0, T], {H}^{s}_\H)$ and the Cauchy data satisfy
$(u_0,u_1)\in {H}^{s+1}_\H \times {H}^{s}_\H$,
then the Cauchy problem \eqref{CPa} has a unique solution
$u\in C([0,T],{H}^{s+1}_\H) \cap
C^1([0,T],{H}^{s}_\H)$ which satisfies the estimate
\begin{equation}
\label{case_1_last-est}
\|u(t,\cdot)\|_{{H}^{s+1}_\H}^2+\| \partial_t u(t,\cdot)\|_{{H}^s_\H}^2\leq
C (\| u_0\|_{{H}^{s+1}_\H}^2+\|u_1\|_{{H}^{s}_\H}^2+\|f\|_{C([0, T], {H}^{s}_\H)}^2).
\end{equation}
\end{thm}

As we have mentioned in the introduction, already in the setting of partial differential equations in $\mathbb R^{n}$, in the cases when $a$ is H\"older or non strictly positive, the well-posedness in the spaces of smooth functions or in the spaces of distributions fail. For example, it is possible to find smooth Cauchy data, taking also $f=0$, so that the Cauchy problem \eqref{CPa} would not have solutions in spaces of distributions, or a solution would exist but would not be unique - we refer to \cite{Colombini-Spagnolo:Acta-ex-weakly-hyp, Colombini-Jannelli-Spagnolo:Annals-low-reg} for respective examples. Therefore, already in such setting Gevrey spaces as well as spaces of ultradistributions appear naturally. Therefore, it is also natural to introduce these spaces in our setting.

It will be convenient to also use the notation $\Cs_{\H}^{\infty}$ for the space of test functions later on, defined by $\Cs_{\H}^{\infty}:=\cap_{s\geq0} H_{\H}^{s}$.
Then we can define the  $\H$-Gevrey (Roumieu)
space $\gamma^s_\H\subset \Cs_{\H}^{\infty}$ by the condition
\begin{equation}\label{DEF:GevL}
f\in \gamma^s_\H \Longleftrightarrow
\exists A>0:
\| \esp^{A \H^{\frac{1}{2s}}} f\|_{\Sp} <\infty,
\end{equation}
for $0<s<\infty$. The expression on the right hand side will be discussed in more detail in Section \ref{SEC:Prelim}. Similarly, we can define the $\H$-Gevrey (Beurling) spaces by
\begin{equation}\label{DEF:GevB}
f\in \gamma^{(s)}_\H \Longleftrightarrow
\forall A>0:
\| \esp^{A \H^{\frac{1}{2s}}} f\|_{\Sp} <\infty,
\end{equation}
for $0<s<\infty$.

We denote by $\Dis_{s}$ and $\Dis_{(s)}$ the spaces of linear continuous functionals on $\gamma^s_\H$ and $\gamma^{(s)}_\H$, respectively. We call them the Gevrey Roumieu ultradistributions and the Gevrey Beurling ultradistributions, respectively.
For further properties we refer to Section \ref{SEC:Prelim}.

\begin{thm}[Case I.2]
\label{theo_case_2}
Assume that $a(t)\ge a_0>0$ and that $a\in C^\alpha([0,T])$ with $0<\alpha<1$. Then
for initial data and for the source term
\begin{itemize}
\item[(a)] $u_0,u_1\in \gamma^s_\H$, $f\in C([0,T]; \gamma^s_\H)$,
\item[(b)] $u_0,u_1\in \Dis_{(s)}$, $f\in C([0,T]; \Dis_{(s)})$,
\end{itemize}
the Cauchy problem \eqref{CPa} has
a unique solution
\begin{itemize}
\item[(a)] $u\in C^2([0,T]; \gamma^s_\H)$,
\item[(b)] $u\in C^2([0,T]; \Dis_{(s)})$,
\end{itemize}
respectively, provided that
\begin{equation}\label{EQ:Case2-s}
1\le s<1+\frac{\alpha}{1-\alpha}.
\end{equation}
\end{thm}

We now consider the situation when the propagation speed $a(t)$ may become
zero but is regular, i.e. $a\in C^\ell([0,T])$ for $\ell\geq 2$.

\begin{thm}[Case I.3]
\label{theo_case_3}
Assume that $a(t)\ge 0$ and that $a\in C^\ell([0,T])$ with $\ell\ge 2$.
Then for initial data and for the source term
\begin{itemize}
\item[(a)] $u_0,u_1\in \gamma^s_\H$, $f\in C([0,T]; \gamma^s_\H)$,
\item[(b)] $u_0,u_1\in \Dis_{(s)}$, $f\in C([0,T]; \Dis_{(s)})$,
\end{itemize}
the Cauchy problem \eqref{CPa} has a unique solution
\begin{itemize}
\item[(a)] $u\in C^2([0,T]; \gamma^s_\H)$,
\item[(b)] $u\in C^2([0,T]; \Dis_{(s)})$,
\end{itemize}
respectively, provided that
\begin{equation}\label{EQ:s1}
1\le s<1+\frac{\ell}{2}.
\end{equation}
If $a(t)\ge 0$ belongs to $C^\infty([0,T])$ then the Cauchy problem \eqref{CPa} is well-posed as in (a) or (b) for every $s\geq 1$.
\end{thm}

We now consider the case which is complementary to that in Theorem \ref{theo_case_3},
namely, when the propagation speed $a(t)$ may become
zero and is less regular, i.e. $a\in C^\alpha([0,T])$ for $0<\alpha<2$.

\begin{thm}[Case I.4]
\label{theo_case_4}
Assume that $a(t)\ge 0$ and that $a\in C^\alpha([0,T])$ with $0<\alpha<2$.
Then, for initial data and for the source term
\begin{itemize}
\item[(a)] $u_0,u_1\in \gamma^s_\H$, $f\in C([0,T]; \gamma^s_\H)$,
\item[(b)] $u_0,u_1\in \Dis_{(s)}$, $f\in C([0,T]; \Dis_{(s)})$,
\end{itemize}
the Cauchy problem \eqref{CPa} has a unique solution
\begin{itemize}
\item[(a)] $u\in C^2([0,T]; \gamma^s_\H)$,
\item[(b)] $u\in C^2([0,T]; \Dis_{(s)})$,
\end{itemize}
respectively, provided that
\begin{equation}\label{EQ:s2}
1\le s<1+\frac{\alpha}{2}.
\end{equation}
\end{thm}

The proofs of the above theorems will be given in Section \ref{SEC:reduction}.

\smallskip
Analogues of Parts (a) of the above theorems for the wave equation on $\mathbb R^{n}$ go back to Colombini, de Giorgi, and Spagnolo \cite{Colombini-deGiordi-Spagnolo-Pisa-1979}. For higher order hyperbolic equations in $\mathbb R$ the Gevrey well-posedness was considered in
\cite{ColKi:02} and \cite{KS06} under assumptions corresponding to Cases I.2 (a) and I.3 (a), which were extended to $\mathbb R^{n}$ in \cite{GR:11} and \cite{GR13}, respectively.
Other low regularity or multiple characteristics situations were considered in e.g. \cite{Bronshtein:TMMO-1980, Cicognani-Colombini:JDE-2013, Colombini-del-Santo-Reissig:BSM-2003}.
Equations with low regularity coefficients often come up in applications, see e.g.
\cite{Hormann-de-Hoop:AAM-2001, Hormann-de-Hoop:AA-2002}.
We refer to \cite{GR:11,GR13} for the history of the subject for hyperbolic equations on $\mathbb R^{n}$ with time-dependent coefficients, as well as for the sharpness of the orders from the theorems above in the case of $\mathbb R^{n}$.
The mathematical analysis of hyperbolic equations with discontinuous coefficients goes back to Hurd and Sattinger \cite{Hurd-Sattinger}. We refer to \cite{GR15b} for the historical review of this topic.






\section{Examples}
\label{SEC:examples}

In this section as an illustration we give several examples of the settings where our results are applicable. Of course, there are many other examples, here we collect the ones for which different types of partial differential equations have particular importance.
We first discuss self-adjoint, and then non-self-adjoint operators.

\subsection{Landau Hamiltonian in 2D}
\label{SEC:LH2d}

First, we describe the setting of the Landau Hamiltonian in 2D. Here, the results of this paper partially recover and also extend the results obtained in \cite{RT16a}.
More precisely, in \cite{RT16a} we considered the magnetic and electric fields of the operator separately, thus treating a more general model in the particular case of the Landau Hamiltonian. On the other hand, in \cite{RT16a} we obtained results corresponding to cases (I.1) and (II.1) only, not dealing with coefficients leading to the appearance of Gevrey type spaces. Therefore, the results of this paper extend those in \cite{RT16a} in the direction of H\"older propagation speeds as well as allowing more general Cauchy data and source terms.

We recall that the dynamics of a particle with charge $e$ and mass $m_{\ast}$ on the Euclidean $xy$--plane in the presence of the perpendicular constant homogeneous magnetic field is described by the  Hamiltonian operator
\begin{equation} \label{eq:Hamiltonian}
\mathcal H_{0}:=\frac{1}{2m_{\ast}} \p{i h \nabla-\frac{e}{c}\A}^{2},
\end{equation}
where $h$ denotes Planck's constant, $c$ is the speed of light and
$i$ the imaginary unit.
In the sequel we can set $m_{\ast}=e=c=h=1$.
With the symmetric gauge
$$
\A=-\frac{\mathrm r}{2}\times2\mathrm B=\p{- B y, B x} ,
$$
where $\mathrm r=(x, y) \in \mathbb R^{2}$,
and $2B>0$ the strength of the
magnetic field, one obtains the Landau Hamiltonian
\begin{equation} \label{eq:LandauHamiltonian}
\H:=\frac{1}{2} \p{\p{i\frac{\partial}{\partial x}-B
y}^{2}+\p{i\frac{\partial}{\partial y}+B x}^{2}},
\end{equation}
acting on the Hilbert space $L^{2}(\mathbb R^{2})$. The spectrum of
$\H$ consists of infinite number of eigenvalues (called the Euclidean Landau levels)
with
infinite multiplicity of the form
\begin{equation} \label{eq:HamiltonianEigenvalues}
\lambda_{n}=\p{2n+1}B, \,\,\, n=0, 1, 2, \dots \,,
\end{equation}
see \cite{F28, L30}.
Denoting the eigenspace of $\H$ corresponding to the eigenvalue $\lambda_{n}$ by
\begin{equation} \label{eq:HamiltonianEigenspaces}
\mathcal A_{n}(\mathbb R^{2})=\{\varphi\in L^{2}(\mathbb R^{2}), \,\, \H \varphi=\lambda_{n}\varphi\},
\end{equation}
its basis is given by (see \cite{ABGM15, HH13}):
{\small
\begin{equation*}
\label{eq:HamiltonianBasis} \left\{
\begin{split}
e^{1}_{k, n}(x,y)&=\sqrt{\frac{n!}{(n-k)!}}B^{\frac{k+1}{2}}\exp\Big(-\frac{B(x^{2}+y^{2})}{2}\Big)(x+iy)^{k}L_{n}^{(k)}(B(x^{2}+y^{2})), \,\,\, 0\leq k, {}\\
e^{2}_{j, n}(x,y)&=\sqrt{\frac{j!}{(j+n)!}}B^{\frac{n-1}{2}}\exp\Big(-\frac{B(x^{2}+y^{2})}{2}\Big)(x-iy)^{n}L_{j}^{(n)}(B(x^{2}+y^{2})), \,\,\, 0\leq j,
\end{split}
\right.
\end{equation*}}
where $L_{n}^{(\alpha)}$ are the Laguerre polynomials given by
$$
L^{(\alpha)}_{n}(t)=\sum_{k=0}^{n}(-1)^{k}C_{n+\alpha}^{n-k}\frac{t^{k}}{k!}, \,\,\, \alpha>-1.
$$
This basis appears also in many related subjects, such as
complex Hermite polynomials \cite{I16}, in
quantization \cite{ABG12,BG14,CGG10}, time-frequency analysis
\cite{A10}, partial differential equations \cite{G08}, planar
point processes \cite{HH13}, as well as in
the Feynman-Schwinger displacement
operator \cite{P86}. Their perturbations have been investigated in
\cite{K16, N96}, and the asymptotic behaviour of the eigenvalues was
analysed in \cite{S14, KP04, M91, PR07, PRV13, LR14, RT08}.

The results of this paper apply for the Cauchy problem \eqref{CPa} for the operator $\H$ from \eqref{eq:LandauHamiltonian}.

\subsection{Harmonic oscillator}
As a second example in any dimension $d\geq1$, we consider the harmonic oscillator of Quantum Mechanics,
$$
\H:=-\Delta+|x|^{2}, \,\,\, x\in\mathbb R^{d}.
$$
The operator $\H$ is essentially self-adjoint on $C_{0}^{\infty}(\mathbb R^{d})$ with eigenvalues
$$
\lambda_{k}=\sum_{j=1}^{d}(2k_{j}+1), \,\,\, k=(k_{1}, \cdots, k_{d})\in\mathbb N^{d},
$$
and with eigenfunctions
$$
u_{k}(x)=\prod_{j=1}^{d}P_{k_{j}}(x_{j}){\rm e}^{-\frac{|x|^{2}}{2}},
$$
which form an orthogonal system in $L^{2}(\mathbb R^{d})$. Here, $P_{n}(\cdot)$ is the $n$--th order Hermite polynomial, and
$$
P_{n}(t)=c_{n}{\rm e}^{\frac{|t|^{2}}{2}}\left(x-\frac{d}{dt}\right)^{n}{\rm e}^{-\frac{|t|^{2}}{2}},
$$
where $t\in\mathbb R$, and
$$
c_{n}=2^{-n/2}(n!)^{-1/2}\pi^{-1/4}.
$$
For more details on the associated spectral analysis, see for instance \cite{NiRo:10}.

\subsection{Higher dimensional Hamiltonian} \label{SEC:n}

Here we describe a higher dimensional example following \cite{RT16a}.
Let $x=(x_{1},\ldots,x_{2d})\in\mathbb R^{2d}$ and again setting all physical constants to be equal to $1$, in analogy to the case of $d=1$ in \eqref{eq:LandauHamiltonian}, let

\begin{equation} \label{eq:Hamiltonian2d}
\H:=\frac12 \p{i \nabla-\A}^{2},
\end{equation}
where
$$
\A=\p{- B_{1}x_{2}, B_{1}x_{1}, -B_{2}x_{4}, B_{2}x_{3},\ldots,-B_{d}x_{2d},B_{d}x_{2d-1}},
$$
corresponding to the magnetic fields of constant strengths $2B_{l}>0$, $l=1,\ldots,d$.
The essentially self-adjoint operator $\H$ on $C_{0}^{\infty}(\mathbb R^{2d})$ in the Hilbert space
$L^{2}(\mathbb R^{2d})=\otimes_{1}^{d}L^{2}(\mathbb R^{2})$ decomposes as
$$
\H=\mathcal H_{1}\otimes I^{\otimes (d-1)}+I\otimes \mathcal H_{2}\otimes I^{\otimes (d-2)}+\cdots+
 I^{\otimes (d-1)}\otimes \mathcal H_{d},
$$
with self-adjoint 2D operators $\mathcal H_{l}$ on $L^{2}(\mathbb R^{2})$ as in \eqref{eq:LandauHamiltonian}.
Let $k=(k_{1},\ldots, k_{d})\in\mathbb N_{0}^{d}$ be a multi-index.
Then in analogy to \eqref{eq:HamiltonianEigenvalues},
the spectrum of $\H$ consists of the infinitely degenerate eigenvalues
\begin{equation} \label{eq:HamiltonianEigenvalues2d}
\lambda_{k}=\sum_{l=1}^{d} B_{l}(2k_{l}+1),
\end{equation}
with eigenfunctions corresponding to \eqref{eq:HamiltonianBasis}.
In particular, in the isotropic case when $B_{l}=B>0$ for all $l$, for two multi-indices $k,k'\in \mathbb N_{0}^{d}$, if $|k|=|k'|$ then
$\lambda_{k}=\lambda_{k'}$ so that the spectrum of $\H$ consists of eigenvalues of the form $\lambda_{m}=B(2m+1)$ with $m\in\mathbb N_{0}.$
We refer e.g. to \cite{P09} and references therein for more details on the spectral analysis of this case.

\subsection{Regular elliptic boundary value problems}

Let $\H$ be a realisation in $L^{2}(\Omega)$ of a
regular elliptic boundary value problem, i.e. such that
the underlying differential
operator is uniformly elliptic and has smooth coefficients on
an open bounded set $\Omega\subset \mathbb R^{n}$, and that the
boundary conditions determining $\H$ are also regular in
some sense. Suppose that $\H$ is a positive elliptic
operator, so that it has a basis of eigenfunctions
in $L^{2}(\Omega)$.

\subsection{Sums of squares on compact Lie groups}
\label{SEC:exgroups}

Let $G$ be a compact Lie group and let $X_{1},\ldots,X_{k}$ be a basis of left-invariant vector fields satisfying the H\"ormander sums of squares condition. Let
$$\H=-\sum_{j=1}^{k}X_{j}^{2}$$
be the (positive) sub-Laplacian. Then $\H$ has a discrete spectrum which can be related to the spectrum of the bi-invariant Laplacian on $G$, see \cite{GR15} for the estimates, also involving the representations of $G$.

We refer to \cite{GR15} for a discussion on the spectral properties and
their history in this case. The cases (I.1)-(I.4) have been partially analysed in
\cite{GR15}, and the results in cases (II.1)-(II.2) extend them to the case of
less regular propagation speeds.

We refer to \cite{Ruzhansky-Turunen:BOOK} for questions related to the Fourier analysis on compact Lie groups.

\subsection{Weighted sub-Laplacians and sub-Riemannian Ornstein-Uhlenbeck operators on H-type and M\'etivier groups}

Let $G$ be a M\'etivier group and let $d$ and $\nabla_{X}$ be a homogeneous norm and the horizontal gradient on $G$, respectively. Let $\H_{\alpha}$ be the weighted sub-Laplacian associated to the Dirichlet form 
$$f\mapsto\int_{G}|\nabla_{X}f(x)|^{2}e^{-\frac12 d(x)^{\alpha}} dx.$$
If $G=\mathbb R^n$ and $d$ is the Euclidean norm, then for $\alpha=2$, the operator $\H_{2}$ is the classical Ornstein-Uhlenbeck operator. Such operators and their properties have been also intensively studied in sub-Riemannian settings. Thus, in \cite{I12} it was shown that if $G$ is an H-type group and $d$ is the Carnot-Carath\'eodory norm, then the operators $\H_{\alpha}$ have discrete spectrum for $\alpha>1$. If $G$ is a general M\'etivier group and $d$ is the Kaplan norm, then it was shown in \cite{BC16} that $\H_{\alpha}$ has discrete spectrum if and only if $\alpha>2$.
We refer to e.g. \cite{BLU07} for definitions of these groups and the corresponding norms.

\subsection{Operators on manifolds with boundary}

Let $M$ be a manifold with (possibly irregular) boundary $\partial M$.
Let $\H_{0}$ be an operator densely defined in $L^{2}(M)$, with discrete spectrum and the eigenfunctions forming a basis in $L^{2}(M)$.

If $M$ is a closed manifold (i.e. compact without boundary) and $\H_{0}$ is a positive elliptic pseudo-differential operator on $X$ then $\H=|\H_{0}|=\H_{0}$ in both the operator sense and in the sense explained in the introduction. In this case the basis of the eigenfunctions of $\H$ can be chosen to be orthonormal.

The operator does not have to be elliptic, for example, if we take a family $X_{1},\ldots,X_{k}$ of smooth vector fields on $M$ satisfying the H\"ormander condition, such that the necessarily positive spectrum of the operator
$$\H=\sum_{j=1}^{k}X_{j}X_{j}^{*}$$
corresponds to a basis in $L^{2}(M)$, then it satisfies our assumption. In the case of $M$ being a compact Lie group and left-invariant vector fields $X_{j}$, this recaptures the example in Section \ref{SEC:exgroups}.

Another example here may be the operator $\H=i\frac{d}{dx}$ on the manifold $M=[0,1]$ equipped with periodic boundary conditions $f(0)=f(1)$. This can be regarded as a special case of the situation above since in this case $M$ can be identified with the circle.

However, if we take the operator $\H_{0}=i\frac{d}{dx}$ on the manifold $M=[0,1]$ equipped with boundary conditions $hf(0)=f(1)$ for a fixed $h>0$, $h\not=1$, it is no longer self-adjoint.
Its eigenvalues are given by
$$ \lambda_{\xi}=-i\ln h+2\pi\xi,\quad \xi\in\mathbb Z,$$
corresponding to eigenfunctions
$$
e_{\xi}(x) =h^{x}e^{2\pi ix\xi},$$
which are not orthogonal.  We note here that
$\H=|\H_{0}|\not=\H_{0}$ makes sense in the symbolic sense of the introduction but not as a composition of $\H_{0}$ with $\H_{0}^{*}$ which can not be composed having different domains.
The function spaces, for example the Sobolev spaces $H^{s}_{\H}$ and $H^{s}_{\H_{0}}$ are comparable since the asymptotic distribution of $\lambda_{\xi}$ and $|\lambda_{\xi}|$ is the same.

The biorthogonal (nonharmonic) Fourier and symbolic analysis of such operators is still possible and was developed in \cite{RT16,RT16b} in a general setting, to which we refer for further details.

\subsection{Differential operators with strong regular boundary conditions}

We finish the list of examples with another non-self-adjoint operator, following \cite{RT16}.
Let ${\rm O}^{(m)}$ be an ordinary differential operator in
$L^{2}(0, 1)$ of order $m$ generated by the
differential expression
\begin{equation}
l(y)\equiv y^{(m)}(x)+\sum_{k=0}^{m-1}p_{k}(x)y^{(k)}(x), \quad
0<x<1, \label{EQ1A}
\end{equation}
with coefficients
$$
p_{k}\in C^{k}[0,1], \,\,\, k=0,1,\ldots,m-1,
$$
and boundary conditions
\begin{equation}
U_{j}(y)\equiv
V_{j}(y)+\sum\limits_{s=0}^{k_{j}}\int\limits_{0}^{1}y^{(s)}(t)\rho_{js}(t)dt=0,
\quad  j=1, \ldots, m, \label{EQ2A}
\end{equation}
where
$$
V_{j}(y)\equiv
\sum\limits_{s=0}^{j}[\alpha_{js}y^{(k_{s})}(0)+\beta_{js}y^{(k_{s})}(1)],
$$
with $\alpha_{js}$ and $\beta_{js}$ some real numbers, and $\rho_{js}\in
L^{2}(0, 1)$ for all $j$ and $s$.

Furthermore, we suppose that the boundary conditions (\ref{EQ2A})
are normed and strong regular in the sense considered by Shkalikov in \cite{Shkalikov}.
Then it can be shown that the eigenvalues have the same algebraic and
geometric multiplicities and, after a suitable adaption for our case, we have

\begin{thm}[\cite{Shkalikov}]
The eigenfunctions of the operator ${\rm O}^{(m)}$ with strong regular
boundary conditions (\ref{EQ2A}) form a Riesz basis in $L^{2}(0, 1)$.
\label{TH: Shkalikov}
\end{thm}

In the monograph of Naimark \cite{Naimark} the spectral
properties of differential operators generated by the differential
expression (\ref{EQ1A}) with the boundary conditions (\ref{EQ2A})
without integral terms were considered. The statement as in Theorem
\ref{TH: Shkalikov} was established in this setting, with the asymptotic formula
for the Weyl eigenvalue counting function $N(\lambda)$ in the form
\begin{equation}
\label{EQ: DistrAsymp_4} N(\lambda)\sim C \lambda^{1/m} \,\,\,\,
\hbox{as} \,\,\,\, \lambda\rightarrow+\infty.
\end{equation}

\section{Main results, Part II: very weak solutions}
\label{Sec:part2}

We now describe the notion of very weak solutions and formulate the corresponding results for distributions $a\in\Dcal'([0,T])$ and $f\in\Dcal'([0,T])\bar\otimes\Dis_{\H}$.
The first main idea is to start from the distributional coefficient $a$ and the source term $f$ to regularise them by convolution with a suitable mollifier $\psi$ obtaining families of smooth functions $(a_{\eps})_\eps$ and $(f_{\eps})_\eps$, namely
\begin{equation}\label{aEQ:regs}
a_{\eps}=a\ast\psi_{\omega(\eps)}, \,\,\, f_{\eps}=f(\cdot)\ast\psi_{\omega(\eps)},
\end{equation}
where  $\psi_{\omega(\eps)}(t)=\omega(\eps)^{-1}\psi(t/\omega(\eps))$ and $\omega(\eps)$ is a positive function converging to $0$ as $\eps\to 0$ to be chosen later.
Here $\psi$ is a Friedrichs--mollifier, i.e.  $\psi\in C^\infty_0(\mathbb R)$, $\psi\ge 0$ and $\int\psi=1$.
It turns out that the net $(a_{\eps})_\eps$ is $C^\infty$-\emph{moderate}, in the sense that its $C^\infty$-seminorms can be estimated by a negative power of $\eps$. More precisely, we will make use of the following notions of moderateness.

In the sequel, the notation $K\Subset\mathbb R$ means that $K$ is
a compact set in $\mathbb R$.
\begin{defi}
\label{adef_mod_intro}
\leavevmode
\begin{itemize}
\item[(i)]
A net of functions $(f_\eps)_\eps\in C^\infty(\mathbb R)^{(0,1]}$ is said to be $C^\infty$-moderate if for all $K\Subset\mathbb R$ and for all $\alpha\in\mathbb N_{0}$ there exist $N\in\mathbb N_{0}$ and $c>0$ such that
$$
\sup_{t\in K}|\partial^\alpha f_\eps(t)|\le c\eps^{-N-\alpha},
$$
for all $\eps\in(0,1]$.
\item[(ii)] A net of functions $(u_\eps)_\eps\in C^\infty([0,T];{H}^{s}_\H)^{(0,1]}$ is $C^\infty([0,T];{H}^{s}_\H)$-moderate if there exist $N\in\mathbb N_{0}$ and $c_k>0$ for all $k\in\mathbb N_{0}$ such that
$$
\|\partial_t^k u_\eps(t,\cdot)\|_{H^{s}_\H}\le c_k\eps^{-N-k},
$$
for all $t\in[0,T]$ and $\eps\in(0,1]$.
\item[(iii)] We say that a net of functions $(u_\eps)_\eps\in C^\infty([0,T];\Dis_{(s)})^{(0,1]}$ is $C^\infty([0,T];\Dis_{(s)})$-moderate if there exists $\eta>0$ and, for all $p\in\mathbb N_{0}$ there exist $c_p>0$ and $N_{p}>0$ such that
$$
\|\esp^{-\eta\H^{\frac{1}{2s}}} \partial_t^p u_\eps(t,\cdot)\|_{\Sp} \le c_p \eps^{-N_{p}-p},
$$
for all $t\in[0,T]$ and $\eps\in(0,1]$.
\end{itemize}
\end{defi}

We note that the conditions of moderateness are natural in the sense that regularisations of distributions are moderate, namely we can regard
\begin{equation}\label{aEQ:incls}
\textrm{ compactly supported distributions } \mathcal{E}'(\mathbb R)\subset \{C^\infty \textrm{-moderate families}\}
\end{equation}
by the structure theorems for distributions.

Thus, while a solution to the Cauchy problems may not exist in the space of distributions on the left hand side of \eqref{aEQ:incls}, it may still exist (in a certain appropriate sense)
in the space on its right hand side. The moderateness assumption will be crucial allowing to recapture the solution as in \eqref{case_1_last-est} should it exist. However, we note that regularisation with standard Friedrichs mollifiers will not be
sufficient, hence the introduction of a family $\omega(\eps)$ in the above regularisations.

We can now introduce a notion of a `very weak solution' for the Cauchy problem
\begin{equation}\label{aCPa}
\left\{ \begin{split}
\partial_{t}^{2}u(t)+a(t)\H u(t)&=f(t), \; t\in [0,T],\\
u(0)&=u_{0}\in\Sp, \\
\partial_{t}u(0)&=u_{1}\in\Sp.
\end{split}
\right.
\end{equation}

\begin{defi}
\label{adef_vws}\leavevmode
Let $s\in\mathbb R$.
\begin{itemize}
\item[(i)]The net $(u_\eps)_\eps\subset C^\infty([0,T];{H}^{s}_\H)$ is {\em a very weak solution of $H^{s}$-type} of the
Cauchy problem \eqref{aCPa} if there exist
\begin{itemize}
\item[]
$C^\infty$-moderate regularisation $a_{\eps}$ of the coefficient $a$,
\item[]
$C^\infty([0,T];{H}^{s}_\H)$-moderate regularisation $f_{\eps}(t)$ of the source term $f(t)$,
\end{itemize}
such that $(u_\eps)_\eps$ solves the regularised problem
\begin{equation}\label{aCPbb}
\left\{ \begin{split}
\partial_{t}^{2}u_{\eps}(t)+a_{\eps}(t)\H u_{\eps}(t)&=f_{\eps}(t), \; t\in [0,T],\\
u_{\eps}(0)&=u_{0}\in\Sp, \\
\partial_{t}u_{\eps}(0)&=u_{1}\in\Sp,
\end{split}
\right.
\end{equation}
for all $\eps\in(0,1]$, and is $C^\infty([0,T];{H}^{s}_\H)$-moderate.
\item[(ii)] We say that the net $(u_\eps)_\eps\subset C^\infty([0,T]; \Dis_{(s)})$ is {\em a very weak solution of $\Dis_{(s)}$-type} of the
Cauchy problem \eqref{aCPa} if there exist
\begin{itemize}
\item[]
$C^\infty$-moderate regularisation $a_{\eps}$ of the coefficient $a$,
\item[]
$C^\infty([0,T]; \Dis_{(s)})$-moderate regularisation $f_{\eps}(t)$ of the source term $f(t)$,
\end{itemize}
such that $(u_\eps)_\eps$ solves the regularised problem
\eqref{aCPbb}
for all $\eps\in(0,1]$, and is $C^\infty([0,T]; \Dis_{(s)})$-moderate.
\end{itemize}
\end{defi}
We note that according to Theorem \ref{theo_case_1} the regularised Cauchy problem
\eqref{aCPbb} has a unique solution satisfying estimate \eqref{case_1_last-est}.

\smallskip
In \cite{GR15b} Claudia Garetto and the first-named author studied weakly hyperbolic second order equations with time-dependent irregular coefficients, assuming that the coefficients are distributions. For such equations, the authors of  \cite{GR15b} introduced the notion of a `very weak solution' adapted to
the type of solutions that exist for regular coefficients. We now apply a modification of this notion to the Cauchy problem \eqref{aCPa}. In fact, in our particular setting, the condition that the distribution $a$ is nonnegative implies that it has to be a Radon measure. However, we will not be making much use of this observation, especially since we could not make the same conclusion on the behaviour of the source term $f$ with respect to $t$.

In the case of the Landau Hamiltonian with irregular (distributional) electromagnetic fields the Sobolev type very weak solutions have been constructed in \cite{RT16a} where we proved the first part of the following result in that case.

\vspace{2mm}

In the following theorem we assume that $a$ is a nonnegative or a strictly positive distribution. The strict positivity means that there exists a constant $a_{0}>0$ such that $a-a_{0}$ is a positive distribution. In other words,
$$
a\geq a_{0}>0,
$$
where $a\geq a_{0}$ means that $a-a_{0}\geq0$, i.e. $\langle a-a_{0}, \psi\rangle\geq0$ for all $\psi\in C^\infty_0(\mathbb R)$, $\psi\geq 0$.

The main results of this part of the paper can be summarised as the following solvability statement complemented by the uniqueness and consistency in Theorems \ref{atheo_consistency-1} and \ref{atheo_consistency-2}.

\begin{thm}[Existence] \leavevmode
\label{atheo_vws} Let $s\in\mathbb R$.
\begin{itemize}
\item[(II.1)] Let the coefficient $a$ of the Cauchy problem \eqref{aCPa} be a positive distribution with compact support included in $[0,T]$, such that $a\ge a_{0}$ for some constant $a_{0}>0$.
Let the Cauchy data $u_0, u_1$ be in ${H}^{s}_\H$ and the source term $f\in\Dcal'([0,T])\bar\otimes {H}^{s}_\H$.
Then the Cauchy problem \eqref{aCPa} has {\rm a very weak solution of $H^{s}$-type}.
\item[(II.2)] Let the coefficient $a$ of the Cauchy problem \eqref{aCPa} be a nonnegative distribution
with compact support included in $[0,T]$, such that $a\ge 0$.
Let the Cauchy data $u_0, u_1$ be in $\Dis_{(s)}$ and the source term $f\in\Dcal'([0,T])\bar\otimes\Dis_{(s)}$.
Then the Cauchy problem \eqref{aCPa} has {\rm a very weak solution of $\Dis_{(s)}$-type}.
\end{itemize}
\end{thm}

In Theorem \ref{atheo_consistency-1} we show that the very weak solution is unique in an appropriate sense.

\smallskip
But now let us formulate the theorem saying that very weak solutions recapture the classical solutions in the case the latter exist. For example, this happens under conditions of Theorem \ref{theo_case_1} and Theorem \ref{theo_case_3} (b). So, we can compare the solution given by Theorem \ref{theo_case_1} and Theorem \ref{theo_case_3} (b) with the very weak solutions in Theorem \ref{atheo_vws} under assumptions when Theorem \ref{theo_case_1} and Theorem \ref{theo_case_3} (b) hold.

\smallskip
As usual, by $L^{\infty}_{1}([0,T])$ we denote the space of bounded functions on $[0,T]$ with the derivative also in $L^{\infty}$.

\begin{thm}[Consistency-1] \leavevmode
\label{atheo_consistency-2}
\begin{itemize}
\item[(II.1)] Assume that $a\in L^{\infty}_{1}([0,T])$ is such that $a(t)\ge a_0>0$. Let $s\in\mathbb R$ and consider the Cauchy problem
\begin{equation}\label{aConsistency:EQ:1-2}
\left\{ \begin{split}
\partial_{t}^{2}u(t)+a(t)\H u(t)&=f(t), \; t\in [0,T],\\
u(0)&=u_{0}\in\Sp, \\
\partial_{t}u(0)&=u_{1}\in\Sp,
\end{split}
\right.
\end{equation}
with $(u_0,u_1)\in {H}^{1+s}_\H\times {H}^{s}_\H$ and the source term $f\in C([0,T]; {H}^{s}_\H)$. Let $u$ be a very weak solution of $H^{s}$-type of
\eqref{aConsistency:EQ:1-2}. Then for any regularising families $a_{\eps}$ and $f_{\eps}$ in Definition \ref{adef_vws}, any representative $(u_\eps)_\eps$ of $u$ converges in $C([0,T];{H}^{1+s}_\H) \cap
C^1([0,T];{H}^{s}_\H)$ as $\eps\rightarrow0$
to the unique classical solution in $C([0,T];{H}^{1+s}_\H) \cap
C^1([0,T];{H}^{s}_\H)$ of the Cauchy problem \eqref{aConsistency:EQ:1-2}
given by Theorem \ref{theo_case_1}.
\item[(II.2)] Assume that $a\in C^\ell([0,T])$ with $\ell\ge 2$ is such that $a(t)\ge 0$. Let $1\le s<1+\frac{\ell}{2}$ and consider the Cauchy problem \eqref{aConsistency:EQ:1-2}
with $u_0, u_1\in\Dis_{(s)}$ and $f\in C([0,T]; \Dis_{(s)})$. Let $u$ be a very weak solution of $\Dis_{(s)}$-type of
\eqref{aConsistency:EQ:1-2}. Then for any regularising families $a_{\eps}$ and $f_{\eps}$ in Definition \ref{adef_vws}, any representative $(u_\eps)_\eps$ of $u$ converges in $C^2([0,T];\Dis_{(s)})$ as $\eps\rightarrow0$
to the unique classical solution in $C^2([0,T];\Dis_{(s)})$ of the Cauchy problem \eqref{aConsistency:EQ:1-2}
given by Theorem \ref{theo_case_3} (b).
\end{itemize}
\end{thm}

In a similar way, we can prove additional consistency `cases' of Theorem \ref{atheo_consistency-2}, corresponding to Theorem \ref{theo_case_2} (b) and
Theorem \ref{theo_case_4} (b):

\begin{thm}[Consistency-2] \leavevmode
\label{atheo_consistency-3}
\begin{itemize}
\item[(II.3)] Assume that $a(t)\ge a_0>0$ and that $a\in C^\alpha([0,T])$ with $0<\alpha<1$. Let $1\le s<1+\frac{\alpha}{1-\alpha}$ and consider the Cauchy problem
\begin{equation}\label{aConsistency:EQ:1-2-2}
\left\{ \begin{split}
\partial_{t}^{2}u(t)+a(t)\H u(t)&=f(t), \; t\in [0,T],\\
u(0)&=u_{0}\in\Sp, \\
\partial_{t}u(0)&=u_{1}\in\Sp,
\end{split}
\right.
\end{equation}
with $u_0, u_1\in\Dis_{(s)}$ and $f\in C([0,T];\Dis_{(s)})$. Let $u$ be a very weak solution of $\Dis_{(s)}$-type of
\eqref{aConsistency:EQ:1-2-2}. Then for any regularising families $a_{\eps}$ and $f_{\eps}$ in Definition \ref{adef_vws}, any representative $(u_\eps)_\eps$ of $u$ converges in $C^2([0,T]; \Dis_{(s)})$ as $\eps\rightarrow0$
to the unique classical solution in $C^2([0,T];\Dis_{(s)})$ of the Cauchy problem \eqref{aConsistency:EQ:1-2-2}
given by Theorem \ref{theo_case_2} (b).
\item[(II.4)] Assume that $a(t)\ge 0$ and that $a\in C^\alpha([0,T])$ with $0<\alpha<2$. Let $1\le s<1+\frac{\alpha}{2}$ and consider the Cauchy problem \eqref{aConsistency:EQ:1-2-2}
with $u_0, u_1\in\Dis_{(s)}$ and $f\in C([0,T];\Dis_{(s)})$. Let $u$ be a very weak solution of $\Dis_{(s)}$-type of
\eqref{aConsistency:EQ:1-2-2}. Then for any regularising families $a_{\eps}$ and $f_{\eps}$ in Definition \ref{adef_vws}, any representative $(u_\eps)_\eps$ of $u$ converges in $C^2([0,T];\Dis_{(s)})$ as $\eps\rightarrow0$
to the unique classical solution in $C^2([0,T];\Dis_{(s)})$ of the Cauchy problem \eqref{aConsistency:EQ:1-2-2}
given by Theorem \ref{theo_case_4} (b).
\end{itemize}
\end{thm}

The proofs of these results will be given in Section \ref{SEC:proofs2}. The uniqueness of the very weak solutions will be formulated in Theorem \ref{atheo_consistency-1}.

\section{$\H$--Fourier analysis}
\label{SEC:Prelim}

In this section we recall the necessary elements of the global Fourier analysis
that has been developed in \cite{RT16} (also see \cite{RT16b}, and its applications to the spectral properties of operators in \cite{DRT}). The space $\Cs_{\H}^{\infty}:={\rm Dom}({\H}^{\infty})$ is called the space of test
functions for ${\H}$.
Here we define
$$
{\rm Dom}({\H}^{\infty}):=\bigcap_{k=1}^{\infty}{\rm Dom}({\H}^{k}),
$$
where ${\rm Dom}({\H}^{k})$ is the domain of the operator ${\H}^{k}$, in turn defined as
$$
{\rm Dom}({\H}^{k}):=\{f\in\Sp: \,\,\, {\H}^{j}f\in {\rm Dom}({\H}), \,\,\, j=0,
\,1, \, 2, \ldots,
k-1\}.
$$
The Fr\'echet topology of $\Cs_{{\H}}^{\infty}$ is given by the family of semi-norms
\begin{equation}\label{EQ:L-top}
\|\varphi\|_{\Cs^{k}_{{\H}}}:=\max_{j\leq k}
\|{\H}^{j}\varphi\|_{\Sp}, \quad k\in\mathbb N_0,
\; \varphi\in\Cs_{{\H}}^{\infty}.
\end{equation}

Analogously to the operator $\H^*$ ($\Sp$-conjugate to $\H$), we introduce the space $\Cs_{\H^*}^{\infty}:={\rm Dom}((\H^*)^{\infty})$ of test
functions for $\H^*$, and we define
$$
{\rm Dom}((\H^*)^{\infty}):=\bigcap_{k=1}^{\infty}{\rm Dom}((\H^*)^{k}),
$$
where ${\rm Dom}((\H^*)^{k})$ is the domain of the operator $(\H^*)^{k}$, in turn defined as
$$
{\rm Dom}((\H^*)^{k}):=\{f\in\Sp: \,\,\, (\H^*)^{j}f\in {\rm Dom}({\H^*}), \,\,\, j=0,
\,1, \, 2, \ldots,
k-1\}.
$$
The Fr\'echet topology of $C_{{\H^*}}^{\infty}$ is given by the family of semi-norms
\begin{equation}\label{EQ:L-top2}
\|\varphi\|_{\Cs^{k}_{{\H^*}}}:=\max_{j\leq k}
\|(\H^*)^{j}\varphi\|_{\Sp}, \quad k\in\mathbb N_0,
\; \varphi\in C_{{\H^*}}^{\infty}.
\end{equation}

The space $$\Dis_{{\H}}:=\mathcal L(\Cs_{\H^*}^{\infty},
\mathbb C)$$ of linear continuous functionals on
$\Cs_{\H^*}^{\infty}$ is called the space of
${\H}$-distributions.
We can understand the continuity here in terms of the topology
\eqref{EQ:L-top2}.
For
$w\in\Dis_{{\H}}$ and $\varphi\in\Cs_{\H^*}^{\infty}$,
we shall write
$$
w(\varphi)=\langle w, \varphi\rangle.
$$
For any $\psi\in\Cs_{\H}^{\infty}$, the functional
$$
\Cs_{\H^*}^{\infty}\ni \varphi\mapsto (\psi, \varphi)
$$
is an ${\H}$-distribution, which gives an embedding $\psi\in\Cs_{{\H}}^{\infty}\hookrightarrow\Dis_{\H}$.

Now the space $$\Dis_{{\H^*}}:=\mathcal L(\Cs_{\H}^{\infty},
\mathbb C)$$ of linear continuous functionals on
$\Cs_{\H}^{\infty}$ is called the space of
${\H^*}$-distributions.
We can understand the continuity here in terms of the topology
\eqref{EQ:L-top}.
For
$w\in\Dis_{{\H^*}}$ and $\varphi\in\Cs_{\H}^{\infty}$,
we shall also write
$$
w(\varphi)=\langle w, \varphi\rangle.
$$
For any $\psi\in\Cs_{\H^*}^{\infty}$, the functional
$$
\Cs_{\H}^{\infty}\ni \varphi\mapsto (\psi, \varphi)
$$
is an ${\H^*}$-distribution, which gives an embedding $\psi\in\Cs_{{\H^*}}^{\infty}\hookrightarrow\Dis_{\H^*}$.

Since the system of eigenfunctions $\{e_{\xi}: \; \xi\in\ind\}$ of the operator $\H$ is a basis in $\Sp$ then its biorthogonal system $\{e_{\xi}^{\ast}: \; \xi\in\ind\}$ is also a basis in $\Sp$ (see e.g. Bari \cite{bari}, as well as Gelfand \cite{Gelfand:analysis-PDEs-UMN-1963}). Note that the function $e_{\xi}^{\ast}$ is an eigenfunction of the operator $\H^*$ corresponding to the eigenvalue $\overline{\lambda}_{\xi}$ for each $\xi\in\ind$.
They satisfy the orthogonality relations
$$(e_{\xi},e_{\eta}^{*})=\delta_{\xi\eta},$$
where $\delta_{\xi\eta}$ is the Kronecker delta.

Let $\mathcal S(\ind)$ denote the space of rapidly decaying
functions $\varphi:\ind\rightarrow\mathbb C$. That is,
$\varphi\in\mathcal S(\ind)$ if for any $m<\infty$ there
exists a constant $C_{\varphi, m}$ such that
$$
|\varphi(\xi)|\leq C_{\varphi, m}\langle\xi\rangle^{-m}
$$
holds for all $\xi\in\ind$, where we denote
$$\langle\xi\rangle:=(1+|\lambda_{\xi}|)^{1/2}.$$

The topology on $\mathcal
S(\ind)$ is given by the seminorms $p_{k}$, where
$k\in\mathbb N_{0}$ and
$$
p_{k}(\varphi):=\sup_{\xi\in\ind}\langle\xi\rangle^{k}|\varphi(\xi)|.
$$

We now define the $\H$-Fourier transform on $\Cs_{\H}^{\infty}$ as the mapping
$$
(\mathcal F_{\H}f)(\xi)=(f\mapsto\widehat{f}):
\Cs_{\H}^{\infty}\rightarrow\mathcal S(\ind)
$$
by the formula
\begin{equation}
\label{FourierTr}
\widehat{f}(\xi):=(\mathcal F_{\H}f)(\xi)=(f, e_{\xi}^{*}),
\end{equation}
and define the $\H^*$-Fourier transform on $\Cs_{\H^*}^{\infty}$ as the mapping
$$
(\mathcal F_{\H^*}g)(\xi)=(g\mapsto\widehat{g}_{\ast}):
\Cs_{\H^{*}}^{\infty}\rightarrow\mathcal S(\ind)
$$
by the formula
\begin{equation}
\label{FourierTr}
\widehat{g}_{*}(\xi):=(\mathcal F_{\H^{*}}g)(\xi)=(g, e_{\xi}).
\end{equation}

The $\H$-Fourier transform
$\mathcal F_{\H}$ is a bijective homeomorphism from $\Cs_{{\H}}^{\infty}$ to $\mathcal
S(\ind)$.
Its inverse  $$\mathcal F_{\H}^{-1}: \mathcal S(\ind)
\rightarrow \Cs_{\H}^{\infty}$$ is given by
\begin{equation}
\label{InvFourierTr} (\mathcal F^{-1}_{{\H}}h)=\sum_{\xi\in\ind} h(\xi) e_{\xi},\quad
h\in\mathcal S(\ind),
\end{equation}
so that the Fourier inversion formula becomes
\begin{equation}
\label{InvFourierTr0}
f=\sum_{\xi\in\ind} \widehat{f}(\xi)e_{\xi}
\quad \textrm{ for all } f\in\Cs_{{\H}}^{\infty}.
\end{equation}

Also the $\H^*$-Fourier transform
$\mathcal F_{\H^*}$ is a bijective homeomorphism from $\Cs_{{\H^*}}^{\infty}$ to $\mathcal
S(\ind)$.
Its inverse  $$\mathcal F_{\H^*}^{-1}: \mathcal S(\ind)
\rightarrow \Cs_{\H^*}^{\infty}$$ is given by
\begin{equation}
\label{InvFourierTr2} (\mathcal F^{-1}_{{\H^*}}g)=\sum_{\xi\in\ind} g(\xi) e^{*}_{\xi},\quad
g\in\mathcal S(\ind),
\end{equation}
so that the $\H^{*}$-Fourier inversion formula becomes
\begin{equation}
\label{InvFourierTr0}
h=\sum_{\xi\in\ind} \widehat{h}_{*}(\xi)e^{*}_{\xi}
\quad \textrm{ for all } h\in \Cs_{{\H{*}}}^{\infty}.
\end{equation}

The Plancherel's identity takes the form
\begin{equation}\label{EQ:Plancherel}
\|f\|_{\Sp}=\p{\sum_{\xi\in\ind}
\widehat{f}(\xi)\overline{\widehat{f}_{*}(\xi)}}^{1/2}.
\end{equation}

We note that since systems $\{e_\xi\}$ and $\{e^{\ast}_\xi\}$ are Riesz bases,
we can also compare $\Sp$--norms of functions with sums of squares of Fourier coefficients.
The following statement follows from the work of Bari \cite[Theorem 9]{bari}:

\begin{lemma}\label{LEM: FTl2}
There exist constants $k,K,m,M>0$ such that for every $f\in\Sp$
we have
$$
m^2\|f\|_{\Sp}^2 \leq \sum_{\xi\in\ind} |\widehat{f}(\xi)|^2\leq M^2\|f\|_{\Sp}^2
$$
and
$$
k^2\|f\|_{\Sp}^2 \leq \sum_{\xi\in\ind} |\widehat{f}_*(\xi)|^2\leq K^2\|f\|_{\Sp}^2.
$$
\end{lemma}

Hence, Lemma \ref{LEM: FTl2} shows that
$$
\|f\|_{1, \Sp}:=\sqrt{(f, f)}=\p{\sum_{\xi\in\ind}
\widehat{f}(\xi)\overline{\widehat{f}_{*}(\xi)}}^{1/2}
$$
and
$$
\|f\|_{2, \Sp}:=\p{\sum_{\xi\in\ind} |\widehat{f}(\xi)|^2}^{1/2}
$$
and
$$
\|f\|_{3, \Sp}:=\p{\sum_{\xi\in\ind} |\widehat{f}_*(\xi)|^2}^{1/2}
$$
are equivalent norms. Indeed, we could use any of them.

\vspace{3mm}

Now we are going to introduce Sobolev spaces induced by the operator $\H$. For this aim we will use $\|\cdot\|_{2, \Sp}$--norm and, briefly, write $\|\cdot\|_{\Sp}$. In fact, it does not matter what norm we use because, as a result, we get equivalent Sobolev norms.

In general, given a linear continuous operator $L:\Cs^{\infty}_{\H}\to
\Cs^{\infty}_{\H}$
(or even $L:\Cs^{\infty}_{\H}\to\Dis_{\H}$), under the condition that $e_{\xi}$ does not have zeros, we can define its symbol by
$\sigma_{L}(\xi):=e_{\xi}^{-1} (L e_\xi)$.
In this case it holds that
\begin{equation}\label{EQ:T-op}
Lf=\sum_{\xi\in\ind} \sigma_{L}(\xi) \, \widehat{f}(\xi) \, e_\xi.
\end{equation}
The correspondence between operators and symbols is one-to-one.
The quantization \eqref{EQ:T-op} has been extensively studied in
\cite{Ruzhansky-Turunen:BOOK,Ruzhansky-Turunen:IMRN} in the setting of compact Lie groups, and in
\cite{RT16} in the setting of (non-self-adjoint) boundary value problems, to which we
may refer for its properties and for the corresponding symbolic calculus.
The condition that $e_{\xi}$ do not have zeros can be removed in some sense, see \cite{RT16b}.
However, in this paper we do not need such technicalities since we already know the symbols of all the appearing operators.



\smallskip
Consequently, we can also define Sobolev spaces $H^s_\H$ associated to
$\H$. Thus, for any $s\in\Rr$, we set
\begin{equation}\label{EQ:HsL}
H^s_\H:=\left\{ f\in\Dis_{\H}: \H^{s/2}f\in
\Sp\right\},
\end{equation}
with the norm $\|f\|_{H^s_\H}:=\|\H^{s/2}f\|_{\Sp}$, which, using
Lemma \ref{LEM: FTl2}, we understand as
\begin{equation*}\label{EQ:Hsub-norm}
\|f\|_{H^s_\H}:=\|\H^{s/2}f\|_{\Sp}:=
\p{\sum_{\xi\in\ind} |\sigma_\H(\xi)|^{s} |\widehat{f}(\xi)|^{2}}^{1/2},
\end{equation*}
justifying the expression \eqref{EQ:Sob1} since $\sigma_\H(\xi)=\lambda_{\xi}$.

The (Roumieu) Gevrey space $\gamma^s_\H$ mentioned in \eqref{DEF:GevL}
is defined by the formula

\begin{multline}\label{DEF:GevL-2}
f\in \gamma^s_\H \Longleftrightarrow
\exists A>0: \\
\| \esp^{A\H^{\frac{1}{2s}}} f\|_{\Sp}^2 := \sum_{\xi\in\ind}
|\esp^{A |\sigma_{\H}(\xi)|^{\frac{1}{2s}}} \widehat{f}(\xi)|^2
= \sum_{\xi\in\ind}
\esp^{2A|\sigma_{\H}(\xi)|^{\frac{1}{2s}}}
|\widehat{f}(\xi)|^2
<\infty.
\end{multline}

Also, define the Beurling Gevrey space $\gamma^{(s)}_\H$ by the formula

\begin{multline}\label{DEF:GevL-2}
g\in \gamma^{(s)}_\H \Longleftrightarrow
\forall B>0: \\
\| \esp^{B\H^{\frac{1}{2s}}} g\|_{\Sp}^2 := \sum_{\xi\in\ind}
|\esp^{B |\sigma_{\H}(\xi)|^{\frac{1}{2s}}} \widehat{g}(\xi)|^2
= \sum_{\xi\in\ind}
\esp^{2B|\sigma_{\H}(\xi)|^{\frac{1}{2s}}}
|\widehat{g}(\xi)|^2
<\infty.
\end{multline}


In the case when $\H$ is the Laplacian (or, more generally, a positive elliptic pseudo-differential operator) on a closed manifold it was shown in \cite{DR16} that these spaces coincide with the usual Gevrey spaces defined in local coordinates.

We denote by $\Dis_{s}$ and $\Dis_{(s)}$ the spaces of linear continuous functionals on $\gamma^s_\H$ and $\gamma^{(s)}_\H$, respectively. We call them the Gevrey Roumieu ultradistributions and the Gevrey Beurling ultradistributions, respectively.
Then, Proposition 13 in \cite{GR:11} can be easily adapted to our case. Hence, the following Fourier characterisations of duals of $\gamma^{s}_{\H}$ and $\gamma^{(s)}_{\H}$ are valid:

\begin{cor}
\label{Fourier character--1}
We have $u\in\Dis_{s}$ if and only if for any $\delta>0$ there exists $C_{\delta}$ such that
$$
|\widehat{u}(\xi)|\leq C_{\delta} e^{\delta|\sigma_{\H}(\xi)|^{\frac{1}{2s}}}
$$
for all $\xi\in\ind$. Similarly, a real functional $u$ belongs to $\Dis_{(s)}$ if and only if there are constants $\eta>0$ and $C>0$ such that
$$
|\widehat{u}(\xi)|\leq C e^{\eta|\sigma_{\H}(\xi)|^{\frac{1}{2s}}}
$$
for all $\xi\in\ind$.
\end{cor}

Again, let us use the Plancherel identity. The Gevrey Roumieu ultradistributions $\Dis_{s}$ and the Gevrey Beurling ultradistributions $\Dis_{(s)}$ can be characterized by
\begin{multline}\label{DEF:Distr-GevL-1}
u\in\Dis_{s} \Longleftrightarrow
\forall \delta>0: \\
\| \esp^{-\delta\H^{\frac{1}{2s}}} u\|_{\Sp}^2 := \sum_{\xi\in\ind}
|\esp^{-\delta |\sigma_{\H}(\xi)|^{\frac{1}{2s}}} \widehat{u}(\xi)|^2
= \sum_{\xi\in\ind}
\esp^{-2\delta|\sigma_{\H}(\xi)|^{\frac{1}{2s}}}
|\widehat{u}(\xi)|^2
<\infty,
\end{multline}
and
\begin{multline}\label{DEF:Distr-GevL-2}
u\in\Dis_{(s)} \Longleftrightarrow
\exists \eta>0: \\
\| \esp^{-\eta\H^{\frac{1}{2s}}} u\|_{\Sp}^2 := \sum_{\xi\in\ind}
|\esp^{-\eta |\sigma_{\H}(\xi)|^{\frac{1}{2s}}} \widehat{u}(\xi)|^2
= \sum_{\xi\in\ind}
\esp^{-2\eta|\sigma_{\H}(\xi)|^{\frac{1}{2s}}}
|\widehat{u}(\xi)|^2
<\infty,
\end{multline}
respectively.

\section{Proofs of Part I: Theorems \ref{theo_case_1}--\ref{theo_case_4}}
\label{SEC:reduction}

The operator $\H$  has the symbol
\begin{equation}
\sigma_{\H}(\xi)=\lambda_{\xi},
\end{equation}
taking the $\H$-Fourier transform of
\eqref{CPa}, we obtain the collection of Cauchy problems for
$\H$--Fourier coefficients:
\begin{equation}\label{CPa-FC}
\partial_{t}^{2}\widehat{u}(t,\xi)+a(t)\sigma_{\H}(\xi)\widehat{u}(t,\xi)=\widehat{f}(t, \xi),
\; \xi\in\ind.
\end{equation}
The main point of our further analysis is that we can make an individual
treatment of the equations in \eqref{CPa-FC}.
Thus, let us fix $\xi\in\ind$,
we then study the Cauchy problem
\begin{equation}\label{EQ:WE-v}
\partial_{t}^{2} \widehat{u}(t,\xi)+ a(t) \sigma_{\H}(\xi)
 \widehat{u}(t,\xi)=\widehat{f}(t, \xi),\quad
 \widehat{u}(0,\xi)=\widehat{u}_{0}(\xi), \;
 \partial_{t}\widehat{u}(0,\xi)=\widehat{u}_{1}(\xi),
\end{equation}
with $\xi$ being a parameter, and want to derive estimates
for $\widehat{u}(t,\xi)$. Combined with the characterisation \eqref{EQ:Sob1} of
Sobolev spaces this will yield the well-posedness
results for the original Cauchy problem \eqref{CPa}.

By setting
\begin{equation}\label{EQ:SETTING}
\nu^2(\xi):=|\sigma_{\H}(\xi)|,
\end{equation}
the equation in \eqref{EQ:WE-v} can be written as
\begin{equation}
\label{eq_xi}
\partial_{t}^{2} \widehat{u}(t,\xi)+a(t)\nu^2(\xi)
\widehat{u}(t,\xi)=\widehat{f}(t, \xi).
\end{equation}

\vspace{3mm}

We now proceed with a standard reduction to a first order system of this equation and define the corresponding energy. The energy estimates will be given in terms of $t$ and $\nu(\xi)$.

We can now do the natural energy construction for \eqref{eq_xi}. We use the transformation
\[
\begin{split}
V_1&:=i\nu(\xi)\widehat{u},\\
V_2&:= \partial_t \widehat{u}.
\end{split}
\]
It follows that the equation \eqref{eq_xi} can be written as the first order system
\begin{equation}\label{EQ:system}
\partial_t V(t,\xi)=i\nu(\xi) A(t)V(t,\xi)+F(t, \xi),
\end{equation}
where $V$ is the column vector with entries $V_1$ and $V_2$,
$$
F(t, \xi)=\left(
    \begin{array}{c}
      0\\
      \widehat{f}(t, \xi)
     \end{array}
  \right),
$$
and
$$
A(t)=\left(
    \begin{array}{cc}
      0 & 1\\
      a(t) & 0 \\
           \end{array}
  \right).
$$
The initial conditions $\widehat{u}(0,\xi)=\widehat{u}_{0}(\xi)$,
$\partial_{t}\widehat{u}(0,\xi)=\widehat{u}_{1}(\xi)$
are transformed into
\[
V(0,\xi)=\left(
    \begin{array}{c}
      i\nu(\xi) \widehat{u}_0(\xi)\\
      \widehat{u}_{1}(\xi)
     \end{array}
  \right).
\]
Note that the matrix $A$ has eigenvalues $\pm\sqrt{a(t)}$ and symmetriser
$$
S(t)=\left(
    \begin{array}{cc}
      a(t) & 0\\
      0 & 1 \\
           \end{array}
  \right).
$$
By definition of the symmetriser we have that
$$
SA-A^\ast S=0.
$$
It is immediate to see that
\begin{equation}
\label{est_sym}
\min_{t\in[0,T]}(a(t),1)|V|^2\le (SV,V)\le \max_{t\in[0,T]}(a(t),1)|V|^2,
\end{equation}
where $(\cdot,\cdot)$ and $|\cdot|$ denote the inner product and the norm in $\C$, respectively.

\subsection{Case I.1: Proof of Theorem \ref{theo_case_1}}
Since $a(t)>0$, $a\in C^1([0,T])$, it is clear that there exist constants $a_0>0$ and $a_1>0$ such
that
\[
a_0=\min_{t\in[0,T]}a(t)
\; \textrm{ and } \;
a_1=\max_{t\in[0,T]}{a(t)}.
\]
Hence \eqref{est_sym} implies,
\begin{equation}
\label{est_sym_1}
c_0|V|^2=\min(a_0,1)|V|^2\le (SV,V)\le \max(a_1,1)|V|^2=c_1|V|^2,
\end{equation}
with $c_0,c_1>0$.
We then define the energy
$E(t,\xi):=(S(t)V(t,\xi),V(t,\xi)),$ and
\begin{align*}
\partial_t E(t,\xi)&=(\partial_t SV,V)+(S \partial_t V,V)
+(SV, \partial_t V) \\
&=(\partial_t SV,V)+i \nu(\xi) (SAV,V)+(SF,V)-i \nu(\xi) (SV,AV)+(SV,F) \\
&=(\partial_t SV,V)+i \nu(\xi) ((SA-A^{*}S)V,V)+2\mathrm{Re}(SF,V) \\
&=(\partial_t SV,V)+2\mathrm{Re}(SF,V) \\
&\le (\Vert \partial_t S\Vert+1) |V|^2+ \Vert SF \Vert^{2} \\
&\le \mathrm{max}(\Vert \partial_t S\Vert+1, \Vert S \Vert^{2}) (|V|^2+|F|^{2}) \\
&\le C_{1} E(t,\xi)+ C_{2} |F|^{2}
\end{align*}
with some constants $C_{1}$ and $C_{2}$. An application of Gronwall's lemma combined with the estimates \eqref{est_sym_1} implies
\begin{equation}
\label{E_1}
|V|^{2}\leq c_{0}^{-1} E(t,\xi)\leq C_{1} |V_{0}|^{2}+ C_{2} \sup_{0\leq t\leq T}|F(t, \xi)|^{2},
\end{equation}
which is valid for all $t\in[0, T]$ and $\xi$ with `new' constants $C_{1}$ and $C_{2}$ depending on $T$. Hence
\[
\nu^{2}(\xi) |\widehat{u}(t,\xi)|^2+|\partial_t \widehat{u}(t,\xi)|^2
\le C_1'( \nu^{2}(\xi) |\widehat{u}_0(\xi)|^2+|\widehat{u}_1(\xi)|^2+\sup_{0\leq t\leq T}|\widehat{f}(t, \xi)|^{2}).
\]
Recalling the notation
$\nu^{2}(\xi)=|\sigma_{\H}(\xi)|$, this means
\begin{equation}
\label{case_1_est_mn} |\sigma_{\H}(\xi)|
|\widehat{u}(t,\xi)|^2+|\partial_t \widehat{u}(t,\xi)|^2
\le C_1'(|\sigma_{\H}(\xi)|
|\widehat{u}_{0}(\xi)|^2+|\widehat{u}_{1}(\xi)|^2+\sup_{0\leq t\leq T}|\widehat{f}(t, \xi)|^{2})
\end{equation}
for all $t\in[0,T]$, $\xi\in\ind$, with the
constant $C_1'$ independent of $\xi$.
Now we recall that by Plancherel's equality, we have
$$
\|\partial_t u(t,\cdot)\|_{\Sp}^2=\sum_{\xi\in\ind}|\partial_t \widehat{u}(t,\xi)|^{2},
$$
$$
\|\H^{1/2} u(t,\cdot)\|_{\Sp}^2=\sum_{\xi\in\ind} |\sigma_\H(\xi)|
|\widehat{u}(t,\xi)|^{2}
$$
and
$$
\|f\|_{C([0, T], \Sp)}^2=\sum_{\xi\in\ind} \|\widehat{f}(\cdot, \xi)\|_{C[0, T]}^{2},
$$
where
$$
\|\widehat{f}(\cdot, \xi)\|_{C[0, T]}^{2}=\sup_{0\leq t\leq T}| \widehat{f}(t, \xi)|^{2}.
$$
Hence, the estimate \eqref{case_1_est_mn} implies that
\begin{equation}
\label{case_1_last}
\|\H^{1/2} u(t,\cdot)\|_{\Sp}^2+\|\partial_t u(t,\cdot)\|_{\Sp}^2\leq
C (\|\H^{1/2} u_0\|_{\Sp}^2+\|u_1\|_{\Sp}^2+\|f\|_{C([0, T], \Sp)}^2),
\end{equation}
where the constant $C>0$ does not depend on $t\in[0,T]$. More
generally, multiplying \eqref{case_1_est_mn} by powers of
$|\sigma_{\H}(\xi)|$, for any $s$, we get
\begin{multline}
\label{case_1_est_mn2} |\sigma_{\H}(\xi)|^{1+s}
|\widehat{u}(t,\xi)|^2+|\sigma_{\H}(\xi)|^{s} |\partial_t
\widehat{u}(t,\xi)|^2 \\
\le C_1'(|\sigma_{\H}(\xi)|^{1+s}
|\widehat{u}_0(\xi)|^2+|\sigma_{\H}(\xi)|^{s}
|\widehat{u}_1(\xi)|^2
+|\sigma_{\H}(\xi)|^{s}\sup_{0\leq t\leq T}|\widehat{f}(t, \xi)|^{2}).
\end{multline}
Taking the sum over $\xi$ as above, this yields
the estimate \eqref{case_1_last-est}.

\subsection{Case I.2: Proof of Theorem \ref{theo_case_2} (a)}

Now, assume that $a(t)\ge a_0>0$ but here the regularity of $a$ is less than $C^1$, i.e.,  $a\in C^\alpha([0,T])$,
with $0<\alpha<1$. Following the notation \eqref{EQ:SETTING} and
as in \cite{GR15} we look for a solution of the system \eqref{EQ:system}, i.e. of
\begin{equation}
\label{system_A}
\partial_t V(t,\xi)={\rm i}\nu(\xi) A(t, \xi)V(t,\xi)+F(t, \xi)
\end{equation}
with
$$
A(t, \xi)=\left(
    \begin{array}{cc}
      0 & 1\\
      a(t)
      & 0 \\
           \end{array}
  \right)
$$
and
$$
F(t, \xi)=\left(
    \begin{array}{c}
      0\\
      \widehat{f}(t, \xi)
     \end{array}
  \right)
$$
of the following form
\[
V(t,\xi)={\rm e}^{-\rho(t)\nu^{1/s}(\xi)}(\det H)^{-1}HW,
\]
where
$\rho\in C^1([0,T])$ is a real-valued function which will be suitably chosen in the sequel,
$W=W(t,\xi)$ is to be determined,
\[
H(t)=\left(
    \begin{array}{cc}
      1 & 1 \\
      -\lambda(t)
      & \lambda(t) 
         \end{array}
  \right),
\]
and, for $\varphi\in C^\infty_c(\mathbb{R})$, $\varphi\ge 0$ with integral $1$,
\begin{equation}
\label{def_lambdaj}
\lambda(t, \eps)=(\sqrt{a}\ast\varphi_\eps)(t),
\end{equation}
where $\varphi_\eps(t)=\frac{1}{\eps}\varphi(t/\eps).$
By construction, $\lambda$ is smooth in $t\in[0,T]$, and
$$
|\lambda(t, \eps)|\ge\sqrt{a_0},
$$
for all $t\in[0,T]$ and $\eps\in(0,1]$,
$$
|\lambda(t, \eps)-\sqrt{a(t)}|\le C \eps^\alpha
$$
uniformly in $t$ and $\eps$.
By substitution in \eqref{system_A} we get
\begin{align*}
\esp^{-\rho(t)\nu^{1/s}(\xi)}(\det H)^{-1}H\partial_tW &+\esp^{-\rho(t)
\nu^{1/s}(\xi)}(-\rho'(t)\nu^{1/s}(\xi))(\det H)^{-1}HW \\
&-\esp^{-\rho(t)\nu^{1/s}(\xi)}\frac{\partial_t\det H}{(\det H)^2}HW\\
&+\esp^{-\rho(t)\nu^{1/s}(\xi)}(\det H)^{-1}(\partial_tH)W\\
&={\rm i}\nu(\xi) \esp^{-\rho(t)\nu^{1/s}(\xi)}(\det H)^{-1}AHW+F.
\end{align*}
Multiplying both sides of the previous equation by $\esp^{\rho(t)\nu^{1/s}(\xi)}(\det H)H^{-1}$ we get
$$
\partial_tW-\rho'(t)\nu^{1/s}(\xi)W-\frac{\partial_t\det H}{\det H}W + H^{-1}(\partial_t H)W = {\rm i}\nu(\xi) H^{-1}AHW+\esp^{\rho(t)\nu^{1/s}(\xi)}(\det H)H^{-1}F.
$$
Hence,
\begin{equation}
\label{energy}
\begin{split}
\partial_t |W(t,\xi)|^2&=2{\rm Re} (\partial_t W(t,\xi),W(t,\xi))\\
&=2\rho'(t)\nu^{1/s}(\xi)|W(t,\xi)|^2\\
&\,\,\,\,\,\,\,\,\,\,\,\,\,\,\,\,\,\,+2\frac{\partial_t\det H}{\det H}|W(t,\xi)|^2\\
&\,\,\,\,\,\,\,\,\,\,\,\,\,\,\,\,\,\,-2{\rm Re}(H^{-1}\partial_t HW,W)\\
&\,\,\,\,\,\,\,\,\,\,\,\,\,\,\,\,\,\,-2\nu(\xi) {\rm Im} (H^{-1}AHW,W)\\
&\,\,\,\,\,\,\,\,\,\,\,\,\,\,\,\,\,\,+2\esp^{\rho(t)\nu^{1/s}(\xi)}{\rm Re}((\det H)H^{-1}F,W).
\end{split}
\end{equation}
It follows that
\begin{equation}
\label{energy_case2}
\begin{split}
\partial_t |W(t,\xi)|^2&\le 2\rho'(t)\nu^{1/s}(\xi)|W(t,\xi)|^2\\
&\,\,\,\,\,\,\,\,\,\,\,\,\,\,\,\,\,\,+2\biggl |\frac{\partial_t\det H}{\det H}\biggr||W(t,\xi)|^2\\
&\,\,\,\,\,\,\,\,\,\,\,\,\,\,\,\,\,\,+2\Vert H^{-1}\partial_t H\Vert|W(t,\xi)^2|\\
&\,\,\,\,\,\,\,\,\,\,\,\,\,\,\,\,\,\,+\nu(\xi) \Vert H^{-1}AH-(H^{-1}AH)^\ast\Vert |W(t,\xi)|^2\\
&\,\,\,\,\,\,\,\,\,\,\,\,\,\,\,\,\,\,+2\esp^{\rho(t)\nu^{1/s}(\xi)}\Vert(\det H)H^{-1}\Vert |F||W|.
\end{split}
\end{equation}

Now we want to show that for all $T>0$ there exist constants $c_1,c_2>0$ such that
\begin{equation}
\label{case_2_1}
\biggl |\frac{\partial_t\det H}{\det H}\biggr |\le c_1\eps^{\alpha-1},
\end{equation}
\begin{equation}
\label{case_2_2}
\Vert H^{-1}\partial_t H\Vert\le c_2\eps^{\alpha-1},
\end{equation}
\begin{equation}
\label{case_2_3}
\Vert H^{-1}AH-(H^{-1}AH)^\ast\Vert\le c_3\eps^\alpha,
\end{equation}
\begin{equation}
\label{case_2_4}
\Vert (\det H) H^{-1}\Vert\le c_4\eps^\alpha,
\end{equation}
for all $t\in[0,T]$ and $\eps\in(0,1]$.

By the definition
\begin{equation}\label{root_est_1}
\biggl|\frac{\partial_t\det H}{\det H}\biggr|=\biggl |\frac{2\partial_t(\lambda(t, \eps))}{2\lambda(t, \eps)}\biggr|\leq\frac{1}{\sqrt{a_{0}}}|\partial_t\lambda(t, \eps)|.
\end{equation}
For large enough $R$ we have that
\begin{multline}\label{root_est_2}
|\partial_t\lambda(t, \eps)|=\eps^{-1}\biggl|\int\limits_{-R}^{R}\sqrt{a}(\tau)\varphi_{\eps}^{'}(t-\tau)d\tau\biggr|=\biggl|\eps^{-1}\int\limits_{-R}^{R}\sqrt{a}(t-\eps\tau)\varphi^{'}(\tau)d\tau\biggr|\\
=\biggl|\eps^{-1}\int\limits_{-R}^{R}(\sqrt{a}(t-\eps\tau)-\sqrt{a}(t))\varphi^{'}(\tau)d\tau+\eps^{-1}\sqrt{a}(t)\int\limits_{-R}^{R}\varphi^{'}(\tau)d\tau\biggr|\leq C \eps^{-1}\eps^{\alpha}
\end{multline}
for some constant $C$. From \eqref{root_est_1} and \eqref{root_est_2} we conclude \eqref{case_2_1}. Since
\begin{equation}\label{root_est_3}
\begin{split}
\Vert H^{-1}\partial_t H\Vert&=\biggl|\biggl| (\lambda(t, \eps)+\lambda(t, \eps))^{-1}\left(
    \begin{array}{cc}
      -\partial_{t}\lambda(t, \eps) & -\partial_{t} \lambda(t, \eps) \\
      \partial_{t} \lambda(t, \eps) & \partial_{t} \lambda(t, \eps)
         \end{array}
  \right)\biggr|\biggr|\\
  &\leq |2\lambda(t, \eps)|^{-1}(2|\partial_{t} \lambda(t, \eps)|),
\end{split}
\end{equation}
by using \eqref{root_est_2}, we get the estimate \eqref{case_2_2}.

Finally, by the direct calculations, we have
\begin{multline*}
\Vert H^{-1}AH-(H^{-1}AH)^\ast\Vert=\biggr|\biggr| \frac{1}{\lambda {\rm e}^{i r(\xi)}}\left(
    \begin{array}{cc}
      0 & -(a-\lambda^{2}){\rm e}^{2 i r(\xi)} \\
      (a-\lambda^{2}){\rm e}^{2 i r(\xi)} & 0
         \end{array}
  \right)\biggr|\biggr|\\
  =\biggr|\frac{a-\lambda^{2}}{\lambda}\biggr|\le c_3\eps^\alpha,
\end{multline*}
and $\Vert (\det H) H^{-1}\Vert\le |\lambda| \le c_4\eps^\alpha$.

Hence, combining \eqref{case_2_1}, \eqref{case_2_2}, \eqref{case_2_3} and \eqref{case_2_4} with the energy \eqref{energy_case2} we obtain
\begin{multline*}
\partial_t |W(t,\xi)|^2\le (2\rho'(t)\nu^{1/s}(\xi)+c_1\eps^{\alpha-1}+c_2\eps^{\alpha-1}
+c_3\eps^{\alpha}\nu(\xi))
|W(t,\xi)|^2\\+c_{4}\esp^{(\rho(t)-\mu)\nu^{1/s}(\xi)}\varepsilon^{\alpha} |W(t,\xi)|,
\end{multline*}
since
$$
|F|\leq C \esp^{-\mu\nu^{1/s}(\xi)}
$$
for some constants $C$, $\mu>0$. Without loss of generality, we can assume $\nu(\xi) >0$.
Hence, by setting $\eps:=\nu^{-1}(\xi)$ we get
\begin{multline*}
\partial_t |W(t,\xi)|^2\le (2\rho'(t)\nu^{1/s}(\xi)+c'_1 \nu^{1-\alpha}(\xi) +c'_3\nu^{1-\alpha}(\xi))|W(t,\xi)|^2\\+c_{4}\nu^{-\alpha}(\xi)\esp^{(\rho(t)-\mu)\nu^{1/s}(\xi)} |W(t,\xi)|.
\end{multline*}
Set now $\rho(t)=\rho(0)-\kappa t$ with $\rho(0)$ and $\kappa>0$ to be chosen later. Assuming $|W(t,\xi)|\geq1$ (for the case $|W(t,\xi)|\leq1$ the same discussions are valid) and taking $\frac{1}{s}> 1-\alpha$ we have
\begin{multline*}
\partial_t |W(t,\xi)|^2\le (-2\kappa\nu^{1/s}(\xi)+c'_1 \nu^{1-\alpha}(\xi) +c'_3\nu^{1-\alpha}(\xi)+c_{4}\nu^{-\alpha}(\xi)\esp^{(\rho(0)-\mu)\nu^{1/s}(\xi)}) |W(t,\xi)|^{2}.
\end{multline*}
At this point, setting $\rho(0)<\mu$, for sufficiently large $\nu(\xi)$
we conclude that
$$
\partial_t|W(t,\xi)|^2\le 0,
$$
for $t\in[0,T]$ and, for example, without loss of generality, for $\nu(\xi) \ge 1$. Passing now to $V$ we get
\begin{equation}
\label{last_estimate}
\begin{split}
|V(t,\xi)|
&\le\esp^{-\rho(t)\nu^{1/s}(\xi)}\frac{1}{\det H(t)}\Vert H(t)\Vert|W(t,\xi)| \\
&\le
\esp^{-\rho(t)\nu^{1/s}(\xi)}\frac{1}{\det H(t)}\Vert H(t)\Vert|W(0,\xi)|\\
&\le\esp^{(-\rho(t)+\rho(0))\nu^{1/s}(\xi)}\frac{\det H(0)}{\det H(t)}\Vert H(t)\Vert \Vert H^{-1}(0)\Vert |V(0,\xi)|,
\end{split}
\end{equation}
where
\[
\frac{\det H(0)}{\det H(t)}\Vert H(t)\Vert \Vert H^{-1}(0)\Vert\le c'.
\]
This is due to the fact that $\det H(t)$ is a bounded function with $\det H(t)=\lambda_2(t)-\lambda_1(t)\ge 2\sqrt{a_0}$ for all $t\in[0,T]$ and $\eps\in(0,1]$, $\Vert H(t)\Vert \le c$ and $\Vert H^{-1}(0)\Vert\le c$ for all $t\in[0,T]$ and $\eps\in(0,1]$. Concluding, there exists a constant $c'>0$ such that
\[
|V(t,\xi)|\le c'\esp^{(-\rho(t)+\rho(0))\nu^{1/s}(\xi)}|V(0,\xi)|,
\]
for all $\nu(\xi) \ge 1$ and $t\in[0,T]$. It is now clear that choosing $\kappa>0$ small enough
we have that if $|V(0,\xi)|\le c\,\esp^{-\delta\nu^{1/s}(\xi)}$, $c,\delta>0$,
the same kind of an estimate holds for $V(t,\xi)$. We finally go back to $\xi$ and $\widehat{v}(t,\xi)$.
The previous arguments lead to
\[
\nu^{2}(\xi) |\widehat{v}(t,\xi)|^2+|\partial_t \widehat{v}(t,\xi)|^2
\le c'\esp^{(-\rho(t)+\rho(0))\nu^{1/s}(\xi)}\nu^{2}(\xi)|\widehat{v_0}(\xi)|^2+
c'\esp^{(-\rho(t)+\rho(0))\nu^{1/s}(\xi)}|\widehat{v_1}(\xi)|^2.
\]
Since
the initial data are both in $\gamma^s_\H$ we obtain that
\begin{equation}
\label{fin_est_case_2}
\nu^{2}(\xi) |\widehat{v}(t,\xi)|^2+|\partial_t \widehat{v}(t,\xi)|^2
\le c'\esp^{
\kappa T\nu^{1/s}(\xi)}(C_0\esp^{-A_0\nu^{1/s}(\xi)}+C_1\esp^{-A_1\nu^{1/s}(\xi)}),
\end{equation}
for suitable constants $C_0,C_1, A_0,A_1>0$ and $\kappa$ small enough, for $t\in[0,T]$ and all $\nu(\xi) \ge 1$.
The estimate \eqref{fin_est_case_2} implies that under the hypothesis of Case I.2  and for
\[
1\le s<1+\frac{\alpha}{1-\alpha},
\]
the solution $u(t, \cdot)\in\gamma^s_\H$ if the initial data are elements of $\gamma^s_\H$ and the source term is from $C([0, T]; \gamma^s_\H)$.

\subsection{Case I.2: Proof of Theorem \ref{theo_case_2} (b)} The proof  of
Theorem \ref{theo_case_2} (b) follows in an analogous way to deriving the proof of Theorem \ref{theo_case_3} (b) from the proof of Theorem \ref{theo_case_2} (a), and by recalling the characterisation of $\Dis_{(s)}$. So we refer to the proof of Theorem \ref{theo_case_3} (b) for more details.

\subsection{Case I.3: Proof of Theorem \ref{theo_case_3} (a)}
\label{ProofCase3}
We now assume that $a(t)\ge 0$ is of class $C^\ell$ on $[0,T]$ with $\ell\ge 2$.
Adopting the notations of the previous cases we want to study the well-posedness of the system
\eqref{EQ:system}:
it follows that the equation \eqref{eq_xi} can be written as the first order system
\[
\partial_t V(t,\xi)=i \nu(\xi) A(t)V(t,\xi)+F(t, \xi),
\]
where $V$ is the column vector with entries $V_1$ and $V_2$,
$$
F(t, \xi)=\left(
    \begin{array}{c}
      0\\
      \widehat{f}(t, \xi)
     \end{array}
  \right),
$$
and
$$
A(t)=\left(
    \begin{array}{cc}
      0 & 1\\
      a(t) & 0 \\
           \end{array}
  \right).
$$
The initial conditions are
\[
V(0,\xi)=\left(
    \begin{array}{c}
      i\nu(\xi) \widehat{u_0}(\xi)\\
      \widehat{u}_{1}(\xi)
     \end{array}
  \right).
\]
Let $Q_\eps$  be a so-called quasi--symmetriser of $A$, defined by
\[
Q_{\eps}(t):=\left(
    \begin{array}{cc}
      a(t) & 0\\
    0 & 1 \\
           \end{array}
  \right) + \eps^2 \left(
    \begin{array}{cc}
      1 & 0\\
      0 & 0 \\
           \end{array}
  \right).
\]
The general technique of using quasi-symmetrisers in weakly hyperbolic problems goes back to D'Ancona and Spagnolo \cite{DS}. For its adaptation to the situation similar to the one under our consideration we can also refer to \cite{GR13}.

\smallskip
Now let us introduce the energy $E_\eps(t,\xi)=(Q_\eps(t) V(t,\xi),V(t,\xi))$. By direct computations we get
\begin{equation}\label{Energy-Estimate-1}
\partial_t E_\eps(t,\xi)=(\partial_t Q_\eps V,V)+i \nu(\xi)((Q_\eps A-A^\ast Q_\eps) V,V)+(Q_\eps F,V)+(Q_\eps V,F).
\end{equation}
Let us calculate $Q_{\eps}A-A^{*}Q_{\eps}.$ By the direct calculations we get
\begin{equation}
\begin{split}
Q_{\eps}A-A^{*}Q_{\eps}&=\left(
    \begin{array}{cc}
      a+\eps^{2} & 0\\
    0 & 1 \\
           \end{array}
  \right) \left(
    \begin{array}{cc}
      0 & 1\\
      a & 0 \\
           \end{array}
  \right)\\
  &\,\,\,\,\,\,\,\,\,\,\,\,\,\,\,\,\,\,\,\,\,\,\,\,\,
  -\left(
    \begin{array}{cc}
      0 & a\\
      1 & 0 \\
           \end{array}
  \right) \left(
    \begin{array}{cc}
      a+\eps^{2} & 0\\
    0 & 1 \\
           \end{array}
  \right)=  \left(
    \begin{array}{cc}
      0 & \eps^{2}\\
    -\eps^{2} & 0 \\
           \end{array}
  \right).
\end{split}
\end{equation}
This implies
$$
((Q_\eps A-A^\ast Q_\eps) V,V)=\eps^{2}\overline{V_{1}}V_{2}-\eps^{2} V_{1}\overline{V_{2}}
$$
for all $V\in\mathbb C^{2}$. By estimating
$$
|((Q_{\eps} A-A^\ast Q_\eps)V,V)|\leq 2 \eps^{2}|V_{1}| |V_{2}|\leq 2 \eps\sqrt{a+\eps^{2}}|V_{1}| |V_{2}|\leq \eps((a+\eps^{2})V_{1}^{2}+V_{2}^{2})=\eps (Q_\eps V,V),
$$
finally, we can write
\begin{equation}\label{Energy-Estimate-2}
|((Q_{\eps} A-A^\ast Q_\eps)V,V)|\le \eps (Q_\eps V,V),
\end{equation}
for all $\eps\in(0,1)$, $t\in[0,T]$ and $V\in\C^2$.

Now to estimate \eqref{Energy-Estimate-1} we prove first that there exists a constant $C\geq1$ such that
\begin{equation}\label{Energy-Estimate-0}
C^{-1}\eps^{2}|V|^2\le (Q_\eps V,V)\le C |V|^2,
\end{equation}
for all $\eps\in(0,1]$, $t\in[0,T]$ and all non-zero continuous functions $V:[0,T]\times\ind\to \C^{2}$.

By recalling components of $V=(V_{1}, V_{2})$, we have
$$
(Q_\eps V,V)=(a+\eps^{2})V_{1}^{2}+V_{2}^{2}\leq C (V_{1}^{2}+V_{2}^{2})=C |V|^2,
$$
and
$$
(Q_\eps V,V)=(a+\eps^{2})V_{1}^{2}+V_{2}^{2}\geq \eps^{2}V_{1}^{2}+V_{2}^{2}\geq C^{-1} (\eps^{2}V_{1}^{2}+\eps^{2}V_{2}^{2})= C^{-1}\eps^{2}|V|^2
$$
for some constant $C\geq1$.

Note, that for all $\eps\in(0,1]$, $t\in[0,T]$ and all non-zero continuous functions $V:[0,T]\times\ind\to \C^{2}$, we get
\begin{equation}\label{Energy-Estimate-3}
\int_{0}^T\frac{|(\partial_t Q_\eps(t) V(t,\xi),V(t,\xi))|}{(Q_\eps(t) V(t,\xi), V(t,\xi))}\, dt\le C\eps^{-2/\ell},
\end{equation}
for some $C_{1}>0$. For more details on the estimate \eqref{Energy-Estimate-3}, see \cite{GR13}, or \cite{KS06}.

Since $f\in C([0, T]; \gamma^s_\H)$, we have $|\widehat{f}(t,\xi)|\leq C {\rm e}^{-\mu \nu^{1/s}(\xi)}$ for all $t\in [0, T]$ and $\xi$, and for some constants $C$, $\mu>0$, we obtain
$$
|(Q_\eps F,V)+(Q_\eps V,F)|\leq 2 \Vert Q_\eps \Vert |F| |V| \leq C_{1} {\rm e}^{-\mu  \nu^{1/s}(\xi)}|V|
$$
for some constant $C_{1}$. Assuming $|V|\geq1$ ($|V|\leq1$ can be considered in a similar way) and by using \eqref{Energy-Estimate-2} and \eqref{Energy-Estimate-3} in \eqref{Energy-Estimate-1}, and by Gronwall's lemma, we get
\begin{equation}
\label{EE_case_3}
E_\eps(t,\xi)\le E_\eps(0,\xi){\rm e}^{c(\eps^{-2/\ell}+\eps\nu(\xi))},
\end{equation}
for some constant $c>0$, uniformly in $t$, $\xi$ and $\eps$. By setting $\eps^{-2/\ell}=\eps\nu(\xi)$ we arrive at
$$
E_\eps(t,\xi)\le E_\eps(0,\xi)C_T{\rm e}^{C_T \nu^{\frac{1}{\sigma}}(\xi)},
$$
with $\sigma=1+\frac{\ell}{2}$. An application of \eqref{Energy-Estimate-0} yields the estimate
$$
C^{-1}\eps^2|V(t,\xi)|^2\le E_\eps(t,\xi)\le E_\eps(0,\xi)C_T{\rm e}^{C_T\nu^{\frac{1}{\sigma}}(\xi)}
\le C|V(0,\xi)|^2C_T{\rm e}^{C_T\nu^{\frac{1}{\sigma}}(\xi)}
$$
which implies
$$
|V(t,\xi)|\le C_2\nu^{\frac{\ell}{2\sigma}}(\xi){\rm e}^{C\nu^{\frac{1}{\sigma}}(\xi)}|V(0,\xi)|,
$$
for some $C_2>0$, for all $t\in[0,T]$ and for all $\xi$. We now go back to $\widehat{u}(t,\xi)$. Hence, we get
\begin{equation}\label{est_1}
|\widehat{u}(t,\xi)|^2\le C^2 \nu^{\frac{\ell}{\sigma}}(\xi){\rm e}^{2C\nu^{\frac{1}{\sigma}}(\xi)}
(\nu^{2}(\xi)|\widehat{u}_0(\xi)|^2+|\widehat{u}_1(\xi)|^2).
\end{equation}
Recall that the initial data $u_0$ and $u_1$ are elements of $\gamma^s_\H$ and, therefore,
there exist constants $A',C'>0$ such that
\begin{equation}
\label{est_2}
|\esp^{A'|\sigma_{\H}(\xi)|^{\frac{1}{2s}}} \widehat{u}_0(\xi)|\le C',\qquad
|\esp^{A'|\sigma_{\H}(\xi)|^{\frac{1}{2s}}}\widehat{u}_1(\xi)|\le C'.
\end{equation}
Inserting \eqref{est_2} in \eqref{est_1}, taking $s<\sigma$  and $\nu(\xi)$
large enough we conclude that there exist constants $C^{''}>0$ such that
$$
|\esp^{A'|\sigma_{\H}(\xi)|^{\frac{1}{2s}}} \widehat{u}(t,\xi)|^2\le C^{''},
$$
for all $t\in[0,T]$. It follows that
$$
\sum_{\xi\in\ind} |\esp^{\frac{A'}{2}|\sigma_{\H}(\xi)|^{\frac{1}{2s}}} \widehat{u}(t,\xi)|^2<\infty,
$$
i.e. $u(t, \cdot)\in\gamma^s_\H$ provided that
$$
1\le s<\sigma=1+\frac{\ell}{2}.
$$

\subsection{Case I.3: Proof of Theorem \ref{theo_case_3} (b)} Not changing anything in the proof of Theorem \ref{theo_case_3} (a), analogously, in this case we have estimate \eqref{est_1}. Recall the characterisation of $\Dis_{(s)}$. Since $u_0$ and $u_1$ are elements of $\Dis_{(s)}$, $f\in C([0, T]; \Dis_{(s)})$, and, therefore, by Corollary \ref{Fourier character--1}
there exist constants $A_{1},C_{1}>0$ such that
\begin{multline}
\label{est_3}
|\esp^{-A_{1}|\sigma_{\H}(\xi)|^{\frac{1}{2s}}} \widehat{u}_0(\xi)|\le C_{1},\\
|\esp^{-A_{1}|\sigma_{\H}(\xi)|^{\frac{1}{2s}}} \widehat{u}_1(\xi)|\le C_{1}, \\
\sup_{t\in [0, T]}|\esp^{-A_{1}|\sigma_{\H}(\xi)|^{\frac{1}{2s}}} \widehat{f}(t, \xi)|\le C_{1}.
\end{multline}
Inserting \eqref{est_3} in \eqref{est_1}, taking $s<\sigma$  and $\nu(\xi)$
large enough we conclude that there exist constant $C_{2}>0$ such that
$$
|\esp^{-A_{1}|\sigma_{\H}(\xi)|^{\frac{1}{2s}}} \widehat{u}(t,\xi)|^2\le C_{2},
$$
for all $t\in[0,T]$. It follows that there are constants $A, C>0$ such that
$$
|\widehat{u}(t,\xi)|\le C \esp^{A\nu^{\frac{1}{s}}(\xi)} ,
$$
i.e. $u(t, \cdot)\in\Dis_{(s)}$ provided that
$$
1\le s<\sigma=1+\frac{\ell}{2}.
$$
This completes the proof of Theorem \ref{theo_case_3} (b).

\subsection{Case I.4: Proof of Theorem \ref{theo_case_4} (a)}

Now assume $a(t)\ge 0$ and $a\in C^\alpha([0,T])$ with $0<\alpha<2$.
Here the roots
$\pm\sqrt{a(t)}$ can coincide and are not H\"older of order $\alpha$ but of order $\alpha/2$.
For an adaptation of the proof of Theorem \ref{theo_case_2}
we will set that
$a\in C^{2\alpha}([0,T])$, $0<\alpha<1$ and that the roots are from
$C^\alpha$.
Again we seek a solution of the system \eqref{system_A}
in the form
\[
V(t,\xi)={\rm e}^{-\rho(t)\nu^{\frac{1}{s}}(\xi)}(\det H)^{-1}HW,
\]
where
$\rho\in C^1([0,T])$ is a real valued function which will be suitably chosen in the sequel,
\[
H(t)=\left(
    \begin{array}{cc}
      1 & 1 \\
      \lambda_1(t, \eps)
      & \lambda_2(t, \eps) 
         \end{array}
  \right)
\]
and, for $\varphi\in C^\infty_{0}(\mathbb{R})$, $\varphi\ge 0$ with integral $1$,  we set
\begin{equation}
\label{def_lambdaj_4}
\begin{split}
\lambda_1(t, \eps)&=(-\sqrt{a}\ast\varphi_\eps)(t)+\eps^\alpha,\\
\lambda_2(t, \eps)&=(+\sqrt{a}\ast\varphi_\eps)(t)+ 2\eps^\alpha.
\end{split}
\end{equation}
Note that $\lambda_1$ and $\lambda_2$ are smooth in $t\in[0,T]$, and
\begin{equation}\label{est-4-0}
\lambda_2(t, \eps)-\lambda_1(t, \eps)\ge \eps^\alpha,
\end{equation}
for all $t\in[0,T]$ and $\eps\in(0,1]$,
\[
|\lambda_1(t, \eps)+\sqrt{a(t)}|\le c_1\eps^\alpha
\]
and
\[
|\lambda_2(t, \eps)-\sqrt{a(t)}|\le c_2\eps^\alpha,
\]
uniformly in $t$ and $\eps$. In analogy to the Case I.2 we take the energy estimate
\begin{multline}\label{est-4-1}
\partial_t |W(t,\xi)|^2\le 2\rho'(t)\nu^{\frac{1}{s}}(\xi)
|W(t,\xi)|^2+2\biggl |\frac{\partial_t\det H}{\det H}\biggr||W(t,\xi)|^2\\
+2\Vert H^{-1}\partial_t H\Vert|W(t,\xi)^2|+\nu(\xi)\Vert H^{-1}AH-(H^{-1}AH)^\ast\Vert |W(t,\xi)|^2\\
+2\esp^{\rho(t)\nu^{1/s}(\xi)}\Vert(\det H)H^{-1}\Vert |F||W|.
\end{multline}

By using \eqref{root_est_2}, \eqref{est-4-0} and discussions of the proof of the Case I.2, it is easy to show that for all $T>0$ there exist constants $c_1,c_2>0$ such that
\begin{equation}
\label{case_4_1}
\biggl |\frac{\partial_t\det H}{\det H}\biggr |\le c_1\eps^{1},
\end{equation}
\begin{equation}
\label{case_4_2}
\Vert H^{-1}\partial_t H\Vert\le c_2\eps^{-1},
\end{equation}
for all $t\in[0,T]$ and $\eps\in(0,1]$.

Now, let us show that
\begin{equation}
\label{case_4_3}
\Vert H^{-1}AH-(H^{-1}AH)^\ast\Vert\le c\eps^\alpha,
\end{equation}
for some $c$, and for all $t\in[0,T]$ and $\eps\in(0,1]$. By the simple calculations we get
\begin{equation*}
\begin{split}
\Vert H^{-1}AH-(H^{-1}AH)^\ast\Vert&=\Bigl|\Bigl|\frac{ 1 
}{\lambda_{2}-\lambda_{1}}\left(
    \begin{array}{cc}
      0 & (\lambda_{1}^{2}+\lambda_{2}^{2}-2a) \\
      -(\lambda_{1}^{2}+\lambda_{2}^{2}-2a) & 0
         \end{array}
  \right) \Big|\Bigl|\\
  &\leq\eps^{-\alpha} |\lambda_{1}^{2}+\lambda_{2}^{2}-2a|\\
  &=\eps^{-\alpha} |(\lambda_{1}-\sqrt{a})(\lambda_{1}+\sqrt{a})+(\lambda_{2}-\sqrt{a})(\lambda_{2}+\sqrt{a})|\\
  &\leq C (|(\lambda_{1}-\sqrt{a})|+|(\lambda_{2}+\sqrt{a})|).
\end{split}
\end{equation*}
Indeed, there is a sufficiently large $R$ and constants $C_{1}, C_{2}$ such that
\begin{equation}\label{root_est_case-4-1}
\begin{split}
|\lambda_{1}(t, \eps)-\sqrt{a}|&=\biggl|\int\limits_{-R}^{R}(\sqrt{a}(t-\eps\tau)-\sqrt{a}(t))\varphi^{'}(\tau)d\tau+\eps^{2}\biggr|\leq C_{1} \eps^{\alpha},\\
|\lambda_{2}(t, \eps)+\sqrt{a}|&=\biggl|\int\limits_{-R}^{R}(\sqrt{a}(t-\eps\tau)-\sqrt{a}(t))\varphi^{'}(\tau)d\tau+2\eps^{2}\biggr|\leq C_{2} \eps^{\alpha}.
\end{split}
\end{equation}
Then \eqref{case_4_3} holds.

By combining \eqref{case_4_1}, \eqref{case_4_2} and \eqref{case_4_3} for $|W(t,\xi)|^2$  we obtain
\begin{multline}
\partial_t |W(t,\xi)|^2\le (2\rho'(t)\nu^{\frac{1}{s}}(\xi)+c_1\eps^{-1}+c_2\eps^{-1}+c_3\eps^{\alpha}\nu(\xi))|W(t,\xi)|^2\\
+c_{4}\esp^{(\rho(t)-\mu)\nu^{1/s}(\xi)}\varepsilon^{\alpha} |W(t,\xi)|
\end{multline}
for some constants $c_{4}$ and $\mu>0$.

Consider the case $|W(t,\xi)|\geq1$. Again, it is not restrictive to assume that $\nu(\xi)>0$. Setting $\eps:=\nu^{-\gamma}(\xi)$ with
\[
\gamma=\frac{1}{1+\alpha}
\]
we get
\begin{multline*}
\partial_t |W(t,\xi)|^2\le (2\rho'(t)\nu^{\frac{1}{s}}(\xi)+c'_1\nu^{\gamma}(\xi)+c'_3\nu^{1-\gamma\alpha}(\xi)+c'_{4}\esp^{(\rho(t)-\mu)\nu^{1/s}(\xi)}\nu^{-\gamma\alpha}(\xi))|W(t,\xi)|^2
\\
\le (2\rho'(t)\nu^{\frac{1}{s}}(\xi)+C_{1}\nu^{1/(1+\alpha)}(\xi)+C_{2}\esp^{(\rho(t)-\mu)\nu^{1/s}(\xi)}\nu^{-\frac{\alpha}{1+\alpha}})(\xi))|W(t,\xi)|^2.
\end{multline*}
At this point taking
\[
\frac{1}{s}> \frac{1}{1+\alpha}
\]
and $\rho(t)=\rho(0)-\kappa t$ with $\kappa>0$ to be chosen later, for large enough $\nu(\xi)$ and $\rho(0)<\mu$ we conclude that
\[
\partial_t|W(t,\xi)|^2\le 0,
\]
for $t\in[0,T]$ and $\nu(\xi)\ge 1$. Passing now to $V$ and by the same arguments of Case I.2 with
\[
\frac{\det H(0)}{\det H(t)}\Vert H(t)\Vert \Vert H^{-1}(0)\Vert\le c\,\eps^{-\alpha}=c\,\nu^{\gamma\alpha}(\xi)
=c\,\nu^{\frac{\alpha}{1+\alpha}}(\xi)
\]
we conclude that there exists a constant $c'>0$ such that
\[
|V(t,\xi)|\le c'\nu^{\frac{\alpha}{1+\alpha}}(\xi)\esp^{(-\rho(t)+\rho(0))\nu^{\frac{1}{s}}(\xi)}|V(0,\xi)|,
\]
for all $\nu(\xi)\ge 1$ and $t\in[0,T]$.
We finally go back to $\widehat{u}(t,\xi)$. We have
\[
\nu^2(\xi)|\widehat{u}(t,\xi)|^2\le c'\esp^{(-\rho(t)+\rho(0))\nu(\xi)^{\frac{1}{s}}}
(\nu(\xi)^2
|\widehat{u}_0(\xi)|^2+ |\widehat{u}_1(\xi)|^2),
\]
with the constant $c'$ independent of $\xi$.
Multiplying by $\esp^{\delta\nu(\xi)^{\frac{1}{s}}}$, we get
\begin{multline}\label{EQ:est-HS}
|\esp^{\delta |\sigma_{\H}(\xi)|^{\frac{1}{2s}}}\sigma_{\H}(\xi) \widehat{u}(t,\xi)|^2 \\
\leq
c' (|\esp^{(-\rho(t)+\rho(0)+\delta) |\sigma_{\H}(\xi)|^{\frac{1}{2s}}}
\sigma_{\H}(\xi) \widehat{u}_0(\xi)|^2
+|\esp^{(-\rho(t)+\rho(0)+\delta)|\sigma_{\H}(\xi)|^{\frac{1}{2s}}} \widehat{u}_1(\xi)|^2),
\end{multline}
for any $\delta>0$.
Since the initial data are both in $\gamma^{s}_{\H}$, we get that
$$
\sum_{\xi\in\ind} (|\esp^{(\kappa T+\delta)|\sigma_{\H}(\xi)|^{\frac{1}{2s}}}\sigma_{\H}(\xi) \widehat{u}_0(\xi)|^2
+|\esp^{(\kappa T+\delta)|\sigma_{\H}(\xi)|^{\frac{1}{2s}}} \widehat{u}_1(\xi)|^2)<\infty
$$
for some $\delta>0$ if $\kappa$ is small enough.
Taking the same sum $\sum_{\xi\in\ind}$ of the expressions in \eqref{EQ:est-HS},
and using Lemma \ref{LEM: FTl2},  we obtain that
\begin{equation}
\label{fin_est_case_4}
\sum_{\xi\in\ind} |\esp^{\delta |\sigma_{\H}(\xi)|^{\frac{1}{2s}}}\sigma_{\H}(\xi) \widehat{u}(t,\xi)|^2
<\infty,
\end{equation}
for $\kappa$ small enough, for $t\in[0,T]$. This completes
the proof of Theorem \ref{theo_case_4}.

\subsection{Case I.4: Proof of Theorem \ref{theo_case_4} (b)} Similarly to the proof of Theorem \ref{theo_case_3} (b) following from the proof of Theorem \ref{theo_case_4} (a), recalling the characterisation of $\Dis_{(s)}$, we get the proof of Theorem \ref{theo_case_4} (b).

\section{Proofs of Part II}
\label{SEC:proofs2}

We start by proving Theorem \ref{atheo_vws} assuring the existence of very weak solutions.

\subsection{Existence of very weak solutions}

As in Theorem \ref{atheo_vws} we consider two cases.

\smallskip
{\bf Case II.1.} We now assume that  coefficient $a=a(t)$ is a distribution with compact support contained in $[0,T]$. Since the formulation of \eqref{aCPa} in this case might be impossible in the distributional sense due to issues related to the product of distributions, we replace \eqref{aCPa} with a regularised equation. In other words, we regularise $a$ by a convolution with a mollifier in $C^\infty_0(\mathbb R)$ and get nets of smooth functions as coefficients. More precisely, let $\psi\in C^\infty_0(\mathbb R)$, $\psi\ge 0$ with $\int\psi=1$, and let $\omega(\eps)$ be a positive function converging to $0$ as $\eps\to 0$, with the rate of convergence to be specified later. Define
$$
\psi_{\omega(\eps)}(t):=\frac{1}{\omega(\eps)}\psi\left(\frac{t}{\omega(\eps)}\right),
$$
$$
a_{\eps}(t):=(a\ast \psi_{\omega(\eps)})(t), \qquad f_{\eps}(t):=(f(\cdot)\ast \psi_{\omega(\eps)})(t), \qquad t\in[0,T].
$$
Since $a$ is a positive distribution with compact support (hence a Radon measure) and $\psi\in C^\infty_0(\mathbb R)$, $\supp\,\psi\subset K$, $\psi\ge 0$,
identifying the measure $a$ with its density, we can write
\begin{align*}
a_{\eps}(t)&=(a\ast \psi_{\omega(\eps)})(t)=\int\limits_{\mathbb R}a(t-\tau)\psi_{\omega(\eps)}(\tau)d\tau=\int\limits_{\mathbb R}a(t-\omega(\eps)\tau)\psi(\tau)d\tau \\
&=\int\limits_{\textsc{K}}a(t-\omega(\eps)\tau)\psi(\tau)d\tau\geq a_{0} \int\limits_{\textsc{K}}\psi(\tau)d\tau:=\tilde a_{0}>0,
\end{align*}
with a positive constant $\tilde a_{0}>0$ independent of $\eps$.

By the structure theorem for compactly supported distributions, we have that there exist $L_{1}, L_{2}\in\mathbb N$ and $c_{1}, c_{2}>0$ such that
\begin{equation}\label{aEQ: a-q-C-moderate}
|\partial^{k}_{t}a_{\eps}(t)|\le c_{1}\,\omega(\eps)^{-L_{1}-k}, \,\,\, |\partial^{k}_{t}f_{\eps}(t)|\le c_{2}\,\omega(\eps)^{-L_{2}-k},
\end{equation}
for all $k\in\mathbb N_{0}$ and $t\in[0,T]$. We note that the numbers $L_{1}$ and $L_{2}$ may be related to the distributional orders of $a$ and $f$ but we will not be needing such a relation in our proof.

Hence, $a_{\eps}$ and $f_{\eps}$ are $C^\infty$--moderate regularisations of the coefficient $a$ and of the source term $f$.
Now, fix $\eps\in(0,1]$, and consider the regularised problem
\begin{equation}\label{aPrTh2:EQ:1}
\left\{
\begin{split}
\partial_{t}^{2}u_{\eps}(t)+a_{\eps}(t)\H u_{\eps}(t)&=f_{\eps}(t), \; t\in [0,T],\\
u_{\eps}(0)&=u_{0}\in\Sp, \\
\partial_{t}u_{\eps}(0)&=u_{1}\in\Sp,
\end{split}
\right.
\end{equation}
with the Cauchy data satisfying
$(u_0,u_1)\in {H}^{s+1}_\H \times {H}^{s}_\H$, $a_{\eps}\in C^\infty[0, T]$ and also $f_{\eps}\in C^\infty([0, T]; {H}^{s}_\H)$.
Then all discussions and calculations of Theorem \ref{theo_case_1} are valid. Thus by Theorem \ref{theo_case_1} the equation \eqref{aPrTh2:EQ:1} has a unique solution in the space $C([0,T];{H}^{s+1}_\H)\cap C^{1}([0,T];{H}^{s}_\H)$. In fact, this unique solution is from $C^\infty([0,T];{H}^{s}_\H)$. This can be checked by taking in account that $a_{\eps}\in C^\infty([0,T])$ and by differentiating both sides of the equation \eqref{aPrTh2:EQ:1} in $t$ inductively. Applying Theorem \ref{theo_case_1} to the equation \eqref{aPrTh2:EQ:1},
using the inequality
$$
\Vert \partial_{t} S(t,\xi) \Vert \leq C |\partial_{t} a_{\eps}(t)| \leq C \omega(\eps)^{-L-1}
$$
and Gronwall's lemma, we get the estimate
\begin{multline}\label{aES: exp-1}
\|u_{\eps}(t,\cdot)\|_{{H}^{s+1}_\H}^2+\| \partial_t u_{\eps}(t,\cdot)\|_{{H}^s_\H}^2
\\
\leq
C \exp(c\,\omega(\eps)^{-L-1}T) (\| u_0\|_{{H}^{s+1}_\H}^2+\|u_1\|_{{H}^{s}_\H}^2+\|f\|_{C([0, T]; {H}^{s}_\H)}^2),
\end{multline}
where the coefficient $L$ is from \eqref{aEQ: a-q-C-moderate}.

Put $\omega^{-1}(\eps)\sim\log\eps$. Then the estimate \eqref{aES: exp-1} transforms to
$$
\|u_{\eps}(t,\cdot)\|_{{H}^{s+1}_\H}^2+\| \partial_t u_{\eps}(t,\cdot)\|_{{H}^s_\H}^2\leq
C \eps^{-L-1} (\| u_0\|_{{H}^{s+1}_\H}^2+\|u_1\|_{{H}^{s}_\H}^2+\|f\|_{C([0, T]; {H}^{s}_\H)}^2),
$$
with possibly new constant $L$. To simplify the notation we continue denoting them by the same letters.

Now, let us show that there exist $N\in\mathbb N_{0}$, $c>0$ and, for all $k\in\mathbb N_{0}$ there exist $N_k>0$ and $c_k>0$ such that
$$
\|\partial_t^k u_\eps(t,\cdot)\|_{{H}^{s}_\H}\le c_k \eps^{-N-k},
$$
for all $t\in[0,T]$, and $\eps\in(0,1]$.

Applying \eqref{est_sym_1} and \eqref{E_1} to $u_{\eps}$, and by
taking account the properties of $a_{\eps}$, we get
\begin{align*}
|\sigma_{\H}(\xi)| \,
|\widehat{u_\eps}(t,\xi)|^2 &+|\partial_t
\widehat{u_\eps}(t,\xi)|^2 \\
& \le C \eps^{-L-1}(|\sigma_{\H}(\xi)| \,
|\widehat{u}_0(\xi)|^2+|\widehat{u}_1(\xi)|^2+\sup_{t\in[0, T]}|\widehat{f}(t, \xi)|^{2})
\end{align*}
for all $t\in[0,T]$, $\xi\in\ind$, for some $L>0$ with the
constant $C$ independent of $\xi$. Thus, we obtain
$$
\|\partial_t u_\eps(t,\cdot)\|_{{H}^{s}_\H} \le C \eps^{-L-1}, \,\,\, \| u_\eps(t,\cdot)\|_{{H}^{s+1}_\H}\le
C \eps^{-L}.
$$
Acting by the iterations of $\partial_{t}$  on the
equality
$$
\partial_{t}^{2}u_{\eps}(t)=-a_{\eps}(t)\H u_{\eps}(t)+f_{\eps}(t),
$$
and taking it in $\Sp$--norms, we conclude that $u_{\eps}$ is
$C^\infty([0,T];{H}^{s}_\H)$-moderate.

This shows that the Cauchy problem \eqref{aCPa} has a very weak solution.

\vspace{2mm}

{\bf Case II.2.} Repeating discussions of Case II.1, in this case we get that for a nonnegative function $a_{\eps}(t)$ there exist $L\in\mathbb N$ and $c_{1}>0$ such that
\begin{equation}\label{aEQ: a-q-C-moderate2}
|\partial^{k}_{t}a_{\eps}(t)|\le c_{1}\,\omega(\eps)^{-L-k},
\end{equation}
for all $k\in\mathbb N_{0}$ and $t\in[0,T]$, i.e. $a_{\eps}$ and $f_{\eps}$ are $C^\infty$--moderate regularisations of the coefficient $a$ and of the source term $f$.
Fix $\eps\in(0,1]$, and consider the regularised problem
\begin{equation}\label{aPrTh2:EQ:2}
\left\{ \begin{split}
\partial_{t}^{2}u_{\eps}(t)+a_{\eps}(t)\H u_{\eps}(t)&=f_{\eps}(t), \; t\in [0,T],\\
u_{\eps}(0)&=u_{0}\in\Sp, \\
\partial_{t}u_{\eps}(0)&=u_{1}\in\Sp,
\end{split}
\right.
\end{equation}
with the Cauchy data satisfy
$u_0,u_1\in \Dis_{(s)}$ and $a_{\eps}\in C^\infty[0, T]$.
Then all discussions and calculations of Theorem \ref{theo_case_3} (b) are valid. Thus by Theorem \ref{theo_case_3} (b) the equation \eqref{aPrTh2:EQ:2} has a unique solution in the space $u\in C^2([0,T]; \Dis_{(s)})$ for any $s$. In fact, this unique solution is from $C^\infty([0,T]; \Dis_{(s)})$. This can be checked by taking in account that $a_{\eps}\in C^\infty([0,T])$ and by differentiating both sides of the equation \eqref{aPrTh2:EQ:2} in $t$ inductively. Applying Theorem \ref{theo_case_3} (b) to the equation \eqref{aPrTh2:EQ:2},
using the inequality
$$
|\partial_{t} a_{\eps}(t)| \leq C \omega(\eps)^{-L-1},
$$
we get the estimate
\begin{multline}\label{VWS-est_1}
|\widehat{u_{\eps}}(t,\xi)|^2+|\partial_t
\widehat{u_\eps}(t,\xi)|^2\le \\
 \le C^2 \nu^{\frac{\ell}{\sigma}}(\xi){\rm e}^{2C\omega(\eps)^{-L-1}\nu^{\frac{1}{\sigma}}(\xi)}
(\nu^{2}(\xi)|\widehat{u}_{0,\varepsilon}(\xi)|^2+|\widehat{u}_{1,\varepsilon}(\xi)|^2+\sup_{t\in[0, T]}|\widehat{f}(t, \xi)|^{2}).
\end{multline}
By putting $\omega^{-1}(\eps)\thicksim \log \eps$ and repeating as in the proof of Theorem \ref{theo_case_3} (b), from \eqref{VWS-est_1} we conclude that there exists $\eta>0$ and, for $p=0,1$ there exist $c_p>0$ and $N_{p}>0$ such that
\begin{equation}\label{VWS-est_2}
\|\esp^{-\eta\H^{\frac{1}{2s}}} \partial_t^{p} u_\eps(t,\cdot)\|_{\Sp} \le c_p \, \eps^{-N_{p}-p},
\end{equation}
for all $t\in[0,T]$ and $\eps\in(0,1]$. Now, we need to prove that the estimate \eqref{VWS-est_2} holds for all $p\in\mathbb N$. To show this we use the equality
$$
\partial_{t}^{2}u_{\eps}(t)=-a_{\eps}(t)\H u_{\eps}(t)+f_{\eps}(t).
$$
Acting by the iterations of $\partial_{t}$ on the last equality
and using properties of $a_{\eps}$ and the estimate \eqref{VWS-est_2}, we obtain that $u_{\eps}$ is
$C^\infty([0,T]; \Dis_{(s)})$-moderate.

The theorem is proved.

\subsection{Consistency with the classical well--posedness}

Here we show that when the coefficients are
regular enough then the very weak solution coincides with the
classical one: this is the content of Theorem \ref{atheo_consistency-2} which we will prove here.

Moreover, we show that the very weak solution provided by Theorem \ref{atheo_vws}
is unique in an appropriate sense. For the formulation of the uniqueness statement it will be convenient to use the language of Colombeau algebras.

\begin{defi}
\label{adef_negl_net} We say that $(u_\eps)_\eps$ is
\emph{$C^\infty$-negligible} if for all $K\Subset\mathbb R$, for
all $\alpha\in\mathbb N$ and for all $\ell\in\mathbb N$ there exists
a constant $c>0$ such that
$$
\sup_{t\in K}|\partial^\alpha u_\eps(t)|\le c\eps^{\ell},
$$
for all $\eps\in(0,1]$.
\end{defi}
Since we are dealing with time-dependent distributions supported in the interval $[0,T]$, it is sufficient to take $K=[0,T]$ in the above definition.

We now introduce the Colombeau algebra as the quotient
$$
\mathcal G(\mathbb R)=\frac{C^\infty-\text{moderate\,
nets}}{C^\infty-\text{negligible\, nets}}.
$$
For the general analysis of $\mathcal G(\mathbb R)$ we refer to
e.g. Oberguggenberger \cite{Oberguggenberger:Bk-1992}.

\begin{thm}[Uniqueness] \leavevmode
\label{atheo_consistency-1}
\begin{itemize}
\item[(II.1)] Let $a$ be a positive distribution
with compact support included in $[0,T]$, such that $a\ge a_{0}$ for some constant $a_{0}>0$.
Let $(u_0,u_1)\in {H}^{s+1}_\H\times {H}^{s}_\H$ and $f\in
\mathcal G([0,T]; H^s_{\H})$ for some $s\in\mathbb R$.
Then there exists an embedding of the coefficient $a$ into $\mathcal G([0,T])$,
such that the Cauchy problem \eqref{aCPa}, that is
\begin{equation*}
\left\{ \begin{split}
\partial_{t}^{2}u(t)+a(t)\H u(t)&=f(t), \; t\in [0,T],\\
u(0)&=u_{0}\in\Sp, \\
\partial_{t}u(0)&=u_{1}\in\Sp,
\end{split}
\right.
\end{equation*}
has a unique solution $u\in
\mathcal G([0,T]; H^s_{\H})$.

\item[(II.2)] Let $a\geq 0$ be a nonnegative distribution
with compact support included in $[0,T]$.
Let $u_0, u_1\in \Dis_{(s)}$ and $f\in
\mathcal G([0,T]; \Dis_{(s)})$ for some $s\in\mathbb R$.
Then there exists an embedding of the coefficient $a$ into $\mathcal G([0,T])$,
such that the Cauchy problem \eqref{aCPa} has a unique solution $u\in
\mathcal G([0,T]; \Dis_{(s)})$.
\end{itemize}
\end{thm}
\begin{proof} {\bf Case II.1.}
Let us show that by embedding coefficients in the
corresponding Colombeau algebras the Cauchy problem has a unique
solution $u\in\mathcal G([0,T]; H^s_{\H})$. Assume now that the Cauchy problem has another
solution $v\in\mathcal G([0,T]; H^s_{\H})$. At the level of
representatives this means
\begin{equation*}
\left\{ \begin{split}
\partial_{t}^{2}(u_\eps-v_\eps)(t)+a_\eps(t)\H (u_\eps-v_\eps)(t)&=f_{\eps}(t), \\
(u_\eps-v_\eps)(0)&=0,  \\
(\partial_{t}u_\eps-\partial_{t}v_\eps)(0)&=0,
\end{split}
\right.
\end{equation*}
where $f_{\eps}$ is $C^{\infty}([0,T]; H^s_{\H})$--negligible.
The corresponding first order system is
$$
\partial_t\left(
                             \begin{array}{c}
                               W_{1,\eps} \\
                                W_{2,\eps} \\
                             \end{array}
                           \right)
= \left(
    \begin{array}{cc}
      0 & i\H^{1/2}\\
      i a_{\eps}(t)\H^{1/2}& 0 \\
           \end{array}
  \right)
  \left(\begin{array}{c}
                               W_{1,\eps} \\

                               W_{2,\eps} \\
                             \end{array}
                           \right)+\left(\begin{array}{c}
                               0 \\

                               f_\eps \\
                             \end{array}
                           \right),
$$
where $W_{1,\eps}$ and $W_{2,\eps}$ are obtained via the
transformation
$$
W_{1,\eps}=\H^{1/2}(u_\eps-v_\eps),\,\,\,
W_{2,\eps}=\partial_t(u_\eps-v_\eps).
$$
This system will be studied after $\H$--Fourier transform, as a system
of the type
\begin{equation*}
\partial_t V_{\eps}(t,\xi)=i \nu(\xi) A_{\eps}(t, \xi)V_{\eps}(t,\xi)+F_{\eps}(t,\xi),
\end{equation*}
with
$$
F_\eps=\left(\begin{array}{c} 0 \\
    \mathcal{F}_{\H}f_\eps \\
             \end{array}
       \right),
$$
and
$$
A_{\eps}(t, \xi)=\left(
    \begin{array}{cc}
      0 & 1\\
      a_{\eps}(t) & 0 \\
           \end{array}
  \right),
$$
with Cauchy data
$$
V_{\eps}(0, \xi)=\left(
    \begin{array}{cc}
      0\\
      0 \\
           \end{array}
  \right).
$$
For the symmetriser
$$
S_{\eps}(t, \xi)=\left(
    \begin{array}{cc}
      a_{\eps}(t) & 0\\
      0 & 1 \\
           \end{array}
  \right)
$$
 define the energy
$$E_{\eps}(t,\xi):=(S_{\eps}(t, \xi)V_{\eps}(t,\xi),V_{\eps}(t,\xi)).$$
We get
\begin{align*}
\partial_t E_{\eps}(t,\xi)&=(\partial_t S_{\eps}(t, \xi)V_{\eps}(t,\xi),V_{\eps}(t,\xi))+(S_{\eps}(t, \xi)\partial_t V_{\eps}(t,\xi),V_{\eps}(t,\xi))\\
&\,\,\,\,\,\,\,\,\,\,\,\,\,\,\,\,\,\,\,\,\,\,\,\,\,\,\,\,\,\,\,\,\,\,\,\,\,\,\,\,\,\,\,\,\,\,\,\,\,\,\,\,\,\,\,\,\,\,\,\,\,\,\,\,
+(S_{\eps}(t, \xi)V_{\eps}(t,\xi),\partial_t V_{\eps}(t,\xi))\\
&=(\partial_t S_{\eps}(t, \xi)V_{\eps}(t,\xi),V_{\eps}(t,\xi))\\
&+i \nu(\xi)
(S_{\eps}(t, \xi)A_{\eps}(t, \xi)V_{\eps}(t,\xi),V_{\eps}(t,\xi))-i \nu(\xi) (S_{\eps}(t, \xi)V_{\eps}(t,\xi), A_{\eps}(t, \xi)V_{\eps}(t,\xi))\\
&+(S_{\eps}(t, \xi)F_{\eps}(t,\xi),V_{\eps}(t,\xi))+(S_{\eps}(t, \xi)V_{\eps}(t,\xi), F_{\eps}(t,\xi))\\
&=(\partial_t S_{\eps}(t, \xi)V_{\eps}(t,\xi),V_{\eps}(t,\xi))+i \nu(\xi) ((S_{\eps}A_{\eps}-A^{\ast}_{\eps} S_{\eps})(t, \xi)V_{\eps}(t,\xi),V_{\eps}(t,\xi))\\
&+(S_{\eps}(t, \xi)F_{\eps}(t,\xi),V_{\eps}(t,\xi))+(V_{\eps}(t,\xi), S_{\eps}(t, \xi) F_{\eps}(t,\xi))\\
&=(\partial_t S_{\eps}(t, \xi)V_{\eps}(t,\xi),V_{\eps}(t,\xi))+2\textrm{Re}(S_{\eps}(t, \xi)F_{\eps}(t,\xi),V_{\eps}(t,\xi))\\
&\le \Vert \partial_t S_{\eps}\Vert |V_{\eps}(t,\xi)|^2+2\textrm{Re}(S_{\eps}(t, \xi)F_{\eps}(t,\xi),V_{\eps}(t,\xi))\\
&\le \Vert \partial_t S_{\eps}\Vert |V_{\eps}(t,\xi)|^2+2\Vert
S_{\eps}\Vert |F_{\eps}(t,\xi)| |V_{\eps}(t,\xi)|.
\end{align*}
Assuming for the moment that $|V_{\eps}(t,\xi)|>1$, we get the
energy estimate
\begin{align*}
\partial_t E_{\eps}(t,\xi)&\le \Vert \partial_t S_{\eps}\Vert |V_{\eps}(t,\xi)|^2+2\Vert
S_{\eps}\Vert |F_{\eps}(t,\xi)| |V_{\eps}(t,\xi)|\\
&\le (\Vert \partial_t S_{\eps}\Vert +2\Vert
S_{\eps}\Vert |F_{\eps}(t,\xi)|)|V_{\eps}(t,\xi)|^2\\
&\le \left(|\partial_t a_{\eps}(t)|+|a_{\eps}(t)| |F_{\eps}(t,\xi)|\right)|V_{\eps}(t,\xi)|^2\\
&\le c\, \omega(\eps)^{-L-1} E_{\eps}(t,\xi),
\end{align*}
i.e. we obtain
\begin{equation}
\label{aEQ:E_1}
\partial_t E_{\eps}(t,\xi)\le c\, \omega(\eps)^{-L-1} E_{\eps}(t,\xi),
\end{equation}
for some constant $c>0$. By Gronwall's lemma applied to
inequality \eqref{aEQ:E_1} we conclude that for all $T>0$
$$
E_{\eps}(t,\xi)\le \exp(c\, \omega(\eps)^{-L-1} \, T) E_{\eps}(0,\xi).
$$
Hence, inequalities \eqref{est_sym_1} yield
\begin{align*}
c_0|V_{\eps}(t,\xi)|^2\le E_{\eps}(t,\xi)&\le \exp(c\, \omega(\eps)^{-L-1}\, T)  E_{\eps}(0,\xi) \\
&\le \exp(c_1\, \omega(\eps)^{-L-1}\, T) |V_{\eps}(0,\xi)|^2,
\end{align*}
for the constant $c_{1}$ independent of $t\in[0,T]$ and $\xi$.

By putting $\omega^{-1}(\eps)\sim\log\eps$, we get
\begin{align*}
|V_{\eps}(t,\xi)|^2\le c\, \eps^{-L-1} |V_{\eps}(0,\xi)|^2
\end{align*}
for some constant $c$ and some (new) $L$.
Since $|V_{\eps}(0,\xi)|=0$, we have
$$
|V_{\eps}(t,\xi)|\equiv0,
$$
for all $\xi$ and for $t\in[0,T]$.

Now consider the case when $|V_{\eps}(t,\xi)|<1$. Assume that $|V_{\eps}(t,\xi)|\geq c \, \omega(\varepsilon)^{\alpha}$ for some constant $c$ and $\alpha>0$. It means
$$
\frac{1}{|V_{\eps}(t,\xi)|}\leq C \, \omega(\varepsilon)^{-\alpha}.
$$
Then the estimate for the energy becomes
\begin{align*}
\partial_t E_{\eps}(t,\xi)&\le C\, \omega(\varepsilon)^{-L_{1}} E_{\eps}(t,\xi),
\end{align*}
where $L_{1}=L+\textrm{max}\{1, \alpha\}$, and by Gronwall's lemma
$$
|V_{\eps}(t,\xi)|^2\le \exp(C'\, \omega(\varepsilon)^{-L_{1}} \, T) |V_{\eps}(0,\xi)|^2.
$$
And again, by putting $\omega^{-1}(\eps)\sim\log\eps$, we get
$$
|V_{\eps}(t,\xi)|^2\le c'\, \varepsilon^{-L_{1}} |V_{\eps}(0,\xi)|^2
$$
for some $c'$ and some (new) $L_{1}$.
Since $|V_{\eps}(0,\xi)|=0$, we have
$$
|V_{\eps}(t,\xi)|\equiv0,
$$
for all $t\in[0, T]$ and $\xi$.

The last case is when $|V_{\eps}(t,\xi)|\leq c \, \omega(\varepsilon)^{\alpha}$ for some constant $c$ and $\alpha>0$. Thus, the first part is proved.

\vspace{3mm}

 {\bf Case II.2.} Here we will repeat some discussions of the first part but we will also use the quasi-symmetrisers. Now, let us show that by embedding coefficients in the
corresponding Colombeau algebras the Cauchy problem has a unique
solution $u\in\mathcal G([0,T]; \Dis_{(s)})$. Assume now that the Cauchy problem has another
solution $v\in\mathcal G([0,T]; \Dis_{(s)})$. At the level of
representatives this means
\begin{equation*}
\left\{ \begin{split}
\partial_{t}^{2}(u_\eps-v_\eps)(t)+a_\eps(t)\H (u_\eps-v_\eps)(t)&=f_{\eps}(t), \\
(u_\eps-v_\eps)(0)&=0,  \\
(\partial_{t}u_\eps-\partial_{t}v_\eps)(0)&=0,
\end{split}
\right.
\end{equation*}
where $f_{\eps}$ is $C^{\infty}([0,T]; \Dis_{(s)})$--negligible.
The corresponding first order system is
$$
\partial_t\left(
                             \begin{array}{c}
                               W_{1,\eps} \\
                                W_{2,\eps} \\
                             \end{array}
                           \right)
= \left(
    \begin{array}{cc}
      0 & i\H^{1/2}\\
      i a_{\eps}(t)\H^{1/2}& 0 \\
           \end{array}
  \right)
  \left(\begin{array}{c}
                               W_{1,\eps} \\

                               W_{2,\eps} \\
                             \end{array}
                           \right)+\left(\begin{array}{c}
                               0 \\

                               f_\eps \\
                             \end{array}
                           \right),
$$
where $W_{1,\eps}$ and $W_{2,\eps}$ are obtained via the
transformation
$$
W_{1,\eps}=\H^{1/2}(u_\eps-v_\eps),\,\,\,
W_{2,\eps}=\partial_t(u_\eps-v_\eps).
$$
This system will be studied after $\H$--Fourier transform, as a system
of the type
\begin{equation*}
\partial_t V_{\eps}(t,\xi)=i \nu(\xi) A_{\eps}(t, \xi)V_{\eps}(t,\xi)+F_{\eps}(t,\xi),
\end{equation*}
with
$$
F_\eps=\left(\begin{array}{c} 0 \\
    \mathcal{F}_{\H}f_\eps \\
             \end{array}
       \right),
$$
and
$$
A_{\eps}(t, \xi)=\left(
    \begin{array}{cc}
      0 & 1\\
      a_{\eps}(t) & 0 \\
           \end{array}
  \right),
$$
with Cauchy data
$$
V_{\eps}(0, \xi)=\left(
    \begin{array}{cc}
      0\\
      0 \\
           \end{array}
  \right).
$$
For the quasi--symmetriser $Q_{\varepsilon}(t, \delta)$, defined as
$$
Q_{\varepsilon}(t, \delta):=\left(
    \begin{array}{cc}
      a_{\varepsilon}(t) & 0\\
    0 & 1 \\
           \end{array}
  \right) + \delta^2 \left(
    \begin{array}{cc}
      1 & 0\\
      0 & 0 \\
           \end{array}
  \right),
$$
we  define the energy
$$
E_\eps(t,\xi, \delta)=(Q_\eps(t, \delta) V(t,\xi),V(t,\xi)).
$$
By direct computations we get
\begin{align*}\label{VWS-Energy-Estimate-1}
\partial_t E_\eps(t,\xi, \delta)=(\partial_t Q_\eps(t, \delta)V(t,\xi),V(t,\xi))+i \nu(\xi)((Q_\eps A-A^\ast Q_\eps)(t)V,V)\\
+2\textrm{Re}(Q_\eps(t, \delta)F_{\eps}(t,\xi),V_{\eps}(t,\xi)).
\end{align*}
By using properties that were established in the proof of Theorem \ref{theo_case_3} and continuing to discuss as in the first part, from the last equality we conclude that the Cauchy problem \eqref{aCPa} has a unique solution $u\in
\mathcal G([0,T]; \Dis_{(s)})$ for all $s\in\mathbb R$.

It completes the proof of Theorem \ref{atheo_consistency-1}.
\end{proof}

\begin{proof}[Proof of Theorem \ref{atheo_consistency-2}]
{\bf Case II.1.} We now want to compare the classical solution $\widetilde{u}$ given by Theorem \ref{theo_case_1} with the very weak solution $u$ provided by  Theorem \ref{atheo_consistency-2}. By the definition of the classical solution we know that
\begin{equation}\label{aConsistency:EQ:2}
\left\{ \begin{split}
\partial_{t}^{2}\widetilde{u}(t)+a(t)\H \widetilde{u}(t)&=f(t), \\
\widetilde{u}(0)&=u_{0}\in\Sp, \\
\partial_{t}\widetilde{u}(0)&=u_{1}\in\Sp.
\end{split}
\right.
\end{equation}
By the definition of the very weak solution $u$, there exists a representative $(u_\eps)_\eps$ of $u$ such
that
\begin{equation}\label{aConsistency:EQ:3}
\left\{ \begin{split}
\partial_{t}^{2}u_{\eps}(t)+a_\eps(t)\H u_{\eps}(t)&=f_{\eps}(t), \\
u_{\eps}(0)&=u_{0}\in\Sp,  \\
\partial_{t}u_{\eps}(0)&=u_{1}\in\Sp,
\end{split}
\right.
\end{equation}
for suitable embeddings of the coefficient $a$ and of the source term $f$. Noting that for $a\in L^{\infty}_{1}([0,T])$ the
nets $(a_{\eps}-a)_\eps$ is converging to $0$ in
$C([0,T])$, we can rewrite
\eqref{aConsistency:EQ:2} as
\begin{equation}\label{aConsistency:EQ:4}
\left\{ \begin{split}
\partial_{t}^{2}\widetilde{u}(t)+\widetilde{a}_\eps(t)\H \widetilde{u}(t)&=\widetilde{f}_{\eps}(t), \\
\widetilde{u}(0)&=u_{0}\in\Sp,  \\
\partial_{t}\widetilde{u}(0)&=u_{1}\in\Sp,
\end{split}
\right.
\end{equation}
where $\widetilde{f}_\eps\in C([0,T]; H^s_{\H})$ and $\widetilde{a}_\eps$ are other representations of $f$ and $a$. From \eqref{aConsistency:EQ:3} and \eqref{aConsistency:EQ:4} we get that $\widetilde{u}-u_\eps$ solves the Cauchy problem
\begin{equation*}
\left\{ \begin{split}
\partial_{t}^{2}(\widetilde{u}-u_\eps)(t)+a_\eps(t)\H(\widetilde{u}-u_\eps)(t)&=n_{\eps}(t), \\
(\widetilde{u}-u_\eps)(0)&=0,  \\
(\partial_{t}\widetilde{u}-\partial_{t}u_\eps)(0)&=0.
\end{split}
\right.
\end{equation*}
As in the first part of the proof we arrive, after
reduction to a system and by application of the Fourier transform
to estimate $|(\widetilde{V}-V_\eps)(t,\xi)|$ in
terms of $(\widetilde{V}-V_\eps)(0,\xi)$ and the right-hand side
$n_\eps(t)$, to the energy estimate
\begin{align*}
\partial_t E_{\eps}(t,\xi)\le &|\partial_t a_{\eps}(t)| |(\widetilde{V}-V_\eps)(t,\xi)|^2\\
&+2|a_{\eps}(t)| |n_{\eps}(t,\xi)| |(\widetilde{V}-V_\eps)(t,\xi)|.
\end{align*}
Since the coefficients are regular
enough, we simply get
\begin{align*}
\partial_t E_{\eps}(t,\xi)&\le c_{1}\, |(\widetilde{V}-V_\eps)(t,\xi)|^2+c_{2}\, |n_{\eps}(t,\xi)| |(\widetilde{V}-V_\eps)(t,\xi)|.
\end{align*}
Since  $(\widetilde{V}-V_\eps)(0,\xi)=0$ and $n_\eps\to 0$ in
$C([0,T];H^s_{\H})$ and continue to discussing as in Theorem \ref{atheo_consistency-1} we conclude that
$|(\widetilde{V}-V_\eps)(t,\xi)|\leq c\, \omega(\varepsilon)^{\alpha}$ for some constant $c$ and $\alpha>0$. Then we have $u_\eps\to \widetilde{u}$
in $C([0,T];{H}^{1+s}_\H) \cap
C^1([0,T];{H}^{s}_\H)$.
Moreover, since any other
representative of $u$ will differ from $(u_\eps)_\eps$ by a
$C^\infty([0,T];H^s_{\H})$-negligible net, the limit is the
same for any representative of $u$.

\vspace{3mm}

{\bf Case II.2.} This part can be proven as the previous Case II.1 with slight modifications.
\end{proof}

\bigskip
Conflict of Interest: The authors declare that they have no conflict of interest.

\end{document}